\documentclass[11pt]{article}

\usepackage{amssymb,amsthm,amsmath,mathrsfs,fullpage, hyperref}
\usepackage{tikz, multirow, tabularx, bm}
\usetikzlibrary{matrix}
\usepackage{pgfplots}

\newcommand{\st}{\mbox{st}}
\DeclareMathOperator{\red}{red}

\makeatletter
\usepackage{tikz}
\newcommand*\circled[2][1.6]{\tikz[baseline=(char.base)]{
    \node[shape=circle, draw, inner sep=0pt, 
        minimum height={\f@size*#1},] (char) {#2};}}
\makeatother

\DeclareMathOperator{\gr}{gr}

\renewcommand\S{\mathcal S}

\newcommand\C{\mathcal C}

\makeatletter
\newcommand{\ostar}{\mathbin{\mathpalette\make@circled\star}}
\newcommand{\make@circled}[2]{%
  \ooalign{$\m@th#1\smallbigcirc{#1}$\cr\hidewidth$\m@th#1#2$\hidewidth\cr}%
}
\newcommand{\smallbigcirc}[1]{%
  \vcenter{\hbox{\scalebox{0.77778}{$\m@th#1\bigcirc$}}}%
}
\makeatother

\newtheorem{theorem}{Theorem}[section]
\newtheorem{lemma}[theorem]{Lemma}\newtheorem{observation}[theorem]{Observation}

\newtheorem{corollary}[theorem]{Corollary}

\theoremstyle{definition}
\newtheorem{definition}[theorem]{Definition}
  
\newtheorem{remark}[theorem]{Remark}
\newtheorem{example}[theorem]{Example}
\usepackage{authblk}
\usepackage{enumerate}

\title{Binary operations on pattern-avoiding cycles}
\date{}
\author[1]{Kassie Archer}
\author[2]{Christina Graves}
\author[1]{Robert P. Laudone}
\affil[1]{{\small Department of Mathematics, United States Naval Academy, Annapolis, MD, 21402}}
\affil[2]{{\small Department of Mathematics, University of Texas at Tyler, Tyler, TX, 75799}}

\affil[ ]{{\small Email: karcher@usna.edu, cgraves@uttyler.edu, laudone@usna.edu }}

\begin{document}

\maketitle

\begin{abstract}
Suppose $c_n(\sigma)$ denotes the number of cyclic permutations in $\S_n$ that avoid a pattern $\sigma$. 
In this paper, we define partial groupoid structures on cyclic pattern-avoiding permutations that allow us  to build larger cyclic pattern-avoiding permutations from smaller ones. We use this structure to find recursive lower bounds on $c_n(\sigma)$. These bounds imply
 that $c_n(\sigma)$ has a growth rate of at least 3 for $\sigma\in\{231,312,321\}$ and a  growth rate of at least 2.6 for $\sigma\in\{123,132,213\}$. 
 In the process, we prove (and sometimes improve) a conjecture of B\'{o}na and Cory  that $c_n(\sigma)\geq 2 c_{n-1}(\sigma)$ for all $\sigma\in\S_3\setminus\{123\}$ and $n\geq 2.$ 
\end{abstract}



\section{Introduction}\label{sec:intro}

Originally posed in 2007 by Richard Stanley, the question of enumerating permutations on $n$ elements that both avoid a pattern $\sigma$ and are composed of only a single $n$-cycle in its cycle decomposition remains open. Related results appear in \cite{AE14,BC2019},  which include enumeration of all cyclic permutations avoiding a pair of patterns with the exception of the pair $\{132, 213\},$ though bounds on this enumeration were found in \cite{H19}. Another related result appears in \cite{ET19}, in which the authors enumerate cyclic permutations that avoid a \textit{consecutive} 123 or 321 pattern.

In \cite[Conjecture 5.2]{BC2019}, B\'{o}na and Cory conjecture that for any $\sigma\in\S_3$, we will have 
\[
2c_{n-1}(\sigma)\leq c_n(\sigma)\leq 4c_{n-1}(\sigma)
\]
where $c_n(\sigma)$ is the number of cyclic permutations of length $n$ that avoid the pattern $\sigma.$ They prove the lower bound for $\sigma=321$, but leave the remaining cases open. 
As part of this paper, we prove (and sometimes improve) the lower bound for $\sigma\neq 123$ and provide an alternative lower bound for the case $\sigma=123$.
 To this end, we establish a partial groupoid structure (i.e., a partially-defined binary operation) on $\sigma$-avoiding cycles, allowing us to build larger permutations from smaller ones (or equivalently, decompose $\sigma$-avoiding cycles into smaller pieces). 
 A similar idea of using a binary operation on pattern-avoiding permutations was used in \cite{G01} to enumerate unimodal (i.e., $\{312,213\}$-avoiding) permutations with a given cycle structure. The authors of \cite{GLW24} similarly defined a partial groupoid structure on $\bigcup_{n\geq 1}\S_n(1324)$ to enumerate a subclass of $1324$-avoiding permutations with additional positional restrictions.

\begin{table}[htp]
    \centering
    \begin{tabular}{|c|c|c|c|c|c|}
    \hline
         $\sigma$& Operation & Definition&  Figure & $\gr(c_n(\sigma))$ & Theorem  \\ \hline
         312 & $\pi=\alpha * \beta$ & Definition~\ref{def:312-star} & Figure~\ref{fig:312 example}&  $\geq 3.04$& Theorem~\ref{thm:312} \\ \hline 
         321 & $\pi=\alpha \odot \beta$ & Definition~\ref{defn:321-combine} & Figure~\ref{fig:321 example}& $\geq 3.17$ & Theorem~\ref{thm:321} \\ \hline 
         123 & $\pi=\alpha \star \beta$ & Definition~\ref{defn:123} & Figure~\ref{fig:123 example}& $\geq 2.60$& Theorem~\ref{thm:123} \\ \hline 
         132 & $\pi=\alpha\ostar\beta$ &Definition~\ref{defn:132}&  Figure~\ref{fig:132 example}& $\geq 2.60$ & Theorem~\ref{thm:132} \\ \hline 
    \end{tabular}
    \caption{For each $\sigma\in\S_3$ (up to symmetry), we define an operation to build larger permutations in $\C_n(\sigma)$ from smaller ones, each demonstrated with a figure in the paper. Each of these operations gives us a lower bound for the growth rate of $c_n(\sigma)=|\C_n(\sigma)|.$}
    \label{tab:1}
\end{table}

We use this decomposition to obtain recursive lower bounds on the enumeration of these patterns, which gives us lower bounds for the growth rates of the sets of these permutations. 
The results from this paper are summarized in Table~\ref{tab:1}. For each $\sigma\in\S_n$, we define a (partial) binary operation on $\sigma$-avoiding cyclic permutations which allows us to build up larger permutations from smaller ones. We use values obtained from small values of $n$ to obtain a recursive lower bound, which in turn gives us a lower bound on the growth rate. We additionally discuss an alternative way to build larger 132-avoiding cycles from smaller ones by augmenting the Dyck path associated to the permutation. We end the paper with some discussion of the results and further directions of study.

\section{Background and Notation}

Let $\S_n$ denote the set of permutations on $[n]=\{1,2,\ldots, n\}$. We write a permutation in its one-line notation as $\pi_1\pi_2\ldots\pi_n$ where $\pi_i=\pi(i)$. We say a permutation $\pi\in \S_n$ \textit{contains} the pattern $\sigma \in \S_k$ if there are indices $i_1<i_2<\cdots < i_k$ so that $\pi_{i_1}\pi_{i_2}\ldots\pi_{i_k}$ is in the same relative order as $\sigma_1\sigma_2\ldots\sigma_k$. We say a permutation $\pi$ \textit{avoids} $\sigma$ if it does not contain it. For example, the permutation $\pi = 3516247$ contains the pattern $132$ since the subsequence $\pi_1\pi_4\pi_6 =364$ is in the same relative order as $132$, and $\pi$ avoids the pattern $321$ since it does not contain any subsequence of length 3 that is decreasing. We denote by $\S_n(\sigma)$ the set of permutations in $\S_n$ that avoid the pattern $\sigma$.  The number of permutations in $\S_n(\sigma)$ for $\sigma\in\S_3$ is well-known to be the $n$-th Catalan number (see for example, \cite{Bona}). 
However, the problem of enumerating those permutations in $\S_n(\sigma)$ that are cyclic remains elusive. 

We call a permutation $\pi\in \S_n$ a \textit{cyclic} permutation, or a \textit{cycle}, if it is composed of a single $n$-cycle when written in its cycle notation (i.e., as a product of disjoint cycles). We write a cycle in its cycle notation as $(1, c_2, c_3, \ldots,c_n)$ where $c_2=\pi_1$ and $c_{i+1}=\pi_{c_i}$ for $2\leq i<n$. For example the permutation $\pi = 4523716 = (1,4,3,2,5,7,6)$ is a cyclic permutation in $\S_7$. We denote by $\C_n$ the set of cyclic permutations on $[n]$ and we denote by $\C_n(\sigma)$ the set of permutations in $\S_n(\sigma)$ that are cyclic, i.e. $\C_n(\sigma) = \S_n(\sigma)\cap\C_n$. Throughout this paper, let let $c_n(\sigma) = |\C_n(\sigma)|$.




There are several symmetries of permutations. For a permutation $\pi =\pi_1\pi_2\ldots \pi_n\in\S_n$, we define
\begin{itemize}\itemsep0em
    \item the \textit{reverse} of  $\pi$ is a permutation $\pi^r$ with $\pi^r_i=\pi_{n+1-i}$, 
    \item the \textit{complement} of $\pi$ is a permutation $\pi^c$ with $\pi^c_i=n+1-\pi_{i}$, and
    \item the \textit{inverse} of $\pi$ is a permutation $\pi^i$ with $\pi^i_j=k$ exactly when $\pi_k=j.$
\end{itemize} 
For example, if $\pi = 4523716,$ then $\pi^r = 6173254,$ $\pi^c=4365172$, and $\pi^i=6341275.$
These operations can be combined; the \textit{reverse-complement} is $\pi^{rc}=(\pi^r)^c$ and the \textit{reverse-complement-inverse} is $\pi^{rci}=(\pi^{rc})^i$. For the previous example, $\pi^{rc}=2715634$ and $\pi^{rci}=3167452.$ We note that $\pi,$ $\pi^i$, $\pi^{rc}$, and $\pi^{rci}$ are all well-known to have the same cycle type, so in particular, $\pi$ is cyclic if and only if one of $\pi^i$, $\pi^{rc}$, and $\pi^{rci}$ is cyclic. One way to see this is from the symmetry of the diagram of the permutation. 

Plotting $(i,\pi_i)$ for each $i$, we obtain the \textit{diagram} of $\pi$. For example, the diagram of $\pi = 34658721$ can be found in the top right corner of Figure~\ref{fig:312 example}, together with its cycle structure. The symmetries described above all correspond to $D_8$ (dihedral) symmetries of this diagram. The inverse, reverse-complement, and reverse-complement-inverse correspond respectively to the reflection on the main diagonal ($y=x$), the $180^\circ$ rotation, and the reflection on the anti-diagonal, each of which will preserve the cycle structure.

Furthermore $\pi$ avoids $\sigma$ if and only if $\pi^s$ avoids $\sigma^s$ for some symmetry $s\in\{r,c,i, rc,rci\}$ or any other combination of symmetries. For this reason, we must have 
\[c_n(312)=c_n(231) \quad \text{ and } \quad c_n(132)=c_n(213)\]
since $231=312^i$ and $132=213^{rc}.$
Therefore, it is enough for us to only consider the four cases $\{123,132,312, 321\}$ in this paper.

Finally, given a subset $S\subseteq [n]$ with $|S|=k$ and a permutation $\pi$ of the numbers in $S$, we write $\red(\pi)$ to denote the permutation in $\S_k$ in the same relative order as $\pi$. For example $\red(739) = 213$ and $\red(924861) = 623541.$ We will often use Greek letters $\alpha$, $\beta$, $\gamma$, and $\delta$ to denote segments of a permutation. For example, if $\alpha=325$ and $\beta =674$, we may write $\pi=\alpha1\beta$ to denote the permutation $3251674.$ In this case $\red(\alpha)=213$ and $\red(1\beta)=1342.$

\section{312-avoiding cycles}

In this section we consider cyclic permutations that avoid the pattern 312.  Note that any permutation that avoids 312 must be of the form $\pi = \gamma_1\ldots \gamma_k 1 \delta_1 \ldots \delta_{n-k-1}$ where $\gamma_i<\delta_j$ for each $i,j$.  In other words, $\pi = \gamma\delta$ where $\gamma\in\S_{k+1}$ is a 312-avoiding permutation that ends in 1, and $\red(\delta)\in\S_{n-k-1}$  is also 312-avoiding. Notice that if $\delta$ is not the empty permutation, $\pi$ will have at least two cycles since all elements of $\{1,2,\ldots, k+1\}$ map to elements in $\{1,2,\ldots, k+1\}$. Therefore the following observation holds. 

\begin{observation}
If $\pi \in \C_n(312)$, then $\pi_n=1$. 
\end{observation}

We obtain slightly more information about cyclic 312-avoiding permutations by considering the position of 2. If  $\pi \in \C_n(312)$ with $n\geq 2$, we can write $\pi = \alpha 2 \beta1 $ with $|\alpha|=k$, $\alpha_i\in [3,k+2]$ for all $1\leq i\leq k$, $|\beta|=n-k-2$, and $\beta_i \in [k+3,  n]$ for all $1\leq i\leq n-k-2$. In fact, if $\pi=\alpha 2 \beta 1 \in \C_n(312)$, then $\alpha21\in\S_{k+2}$ and the reduced permutation $\red(\beta 1)\in\S_{n-k-1}$ are both cyclic as seen in the next lemma. 

\begin{lemma} \label{lem:312Break2}
Let $n\geq 2$ and $\pi \in S_n$ with $\pi = \alpha 2\beta 1$. Then $\pi$ is cyclic and 312-avoiding if and only if the permutations $\alpha21$ and $\red(\beta1)$ are both cyclic and 312-avoiding.
\end{lemma}

\begin{proof}
First suppose $\pi = \alpha2\beta1 \in \C_n(312)$ with $|\alpha|=k$. Since $\pi$ avoids 312, we have $\pi_i\in[2,k+2]$ for all $i\in[1,k+1]$ and $\pi_j\in[k+3,n]$ for all $j\in[k+2,n-1]$. Since all small elements map to small elements and large elements to large elements, these form blocks in the cycle form. In particular, we can write 
\[ \pi = (1, a_1, \ldots, a_{k+1}, b_1, \ldots, b_{n-k-2})\] with $a_i\in[2,k+1]$ for $i\in[1,k]$, $a_{k+1}=k+2,$ $b_j\in[k+3,n-1]$ for $j\in[1,n-k-3],$ and $b_{n-k-2}=n$. 
Thus $\alpha 21 =(1, a_1, \ldots a_{k+1})\in\S_{k+2}$ is cyclic, avoids $312$, and has the property that $\pi_k=2$, and similarly, $\st(\beta 1) = (1,b_1 - k-1, b_2-k-1, \cdots, b_{n-k-2}-k-1)\in\S_{n-k-1}$ is cyclic and avoids 312.

Conversely, suppose $\alpha 21\in\S_{k+2}$ and $\beta 1 \in\S_{n-k-1}$ are both 312-avoiding cyclic permutations with $\alpha21 = (1,a_1, \ldots a_{k+1})$ and $\beta1 = (1, b_1, \ldots b_{n-k-2})$. To see that $\pi = \alpha2\beta1$ is cyclic, we note that the smallest $i$ so that $\pi_i > k+2$ is $i = k+2$. Thus $\pi = (1, a_1, \ldots, a_{k+1}, b_1, \ldots, b_{n-k-2}).$
\end{proof}

\begin{example}
    If we have
    \[\pi=\alpha2\beta1=4\ 3\ 6\ 5\ 2\ 10\ 9\ 11\ 8\ 7\ 13\ 14\ 12\ 1 = (1, 4, 5, 2, 3, 6, 10, 7, 9, 8, 11, 13, 12, 14)\in\S_{14},\]
    then $\alpha21=436521=(1, 4, 5, 2, 3, 6)\in\S_6$ and $\red(\beta1) = 546328971=(1, 5, 2, 4, 3, 6, 8, 7, 9)\in\S_9$ are both cyclic, 312-avoiding permutations.
\end{example}

Though this process gives us a way of building permutations in $\C_n(321)$ from smaller ones, we will give a variation of this process via a binary operator $*.$

Let us define the set $X_{1,n}$ in the following way (the reason for this notation will become clear later):
\[ X_{1,n} = \{2\beta: \beta \in \C_{n-1}(312)\}.\]
Note that by Lemma~\ref{lem:312Break2}, these permutation are cyclic and 312-avoiding. In particular, this is exactly the set of permutations in $\C_{n}(312)$ that begin with 2. As a consequence of Lemma~\ref{lem:312Break2}, we have that $|X_{1,n}| = c_{n-1}(312)$. Since $X_{1,n} \subseteq C_{n}(312),$ we thus have $c_n(312) \geq c_{n-1}(312)$. We will improve this bound greatly by defining similar sets that have no intersection with $X_{1,n}$. 

Note that the permutation $3421$ is in $\C_4(312)$ and is not in the set $X_{1, 4}$. It turns out that inserting $34$ to the front of any 312-avoiding cyclic permutation (and shifting elements appropriately) creates a new 312-avoiding cyclic permutation. In fact, in the next definition and lemma, we will see that if $\alpha$ is any cyclic 312-avoiding permutation of length $m$ that ends in 21, we can insert the first $m-2$ elements of $\alpha$ to the front of any cyclic $312$-avoiding permutation $\beta$ and the resulting permutation will be cyclic and 312-avoiding.




\begin{definition}\label{def:312-star}
Let $\alpha = \alpha_1\alpha_2\cdots\alpha_{m-2}21 \in \C_m(312)$ and $\beta = \beta_1\beta_2 \cdots \beta_{n-1}1 \in \C_n(312)$ for $m,n \geq 2.$  Define the permutation $\alpha * \beta \in S_{m+n-2}$ by
\[ \alpha * \beta = \alpha_1\alpha_2\cdots \alpha_{m-2}\bar{\beta}_1\bar{\beta}_2\cdots \overline{\beta}_{n-1}1\]
where $\bar{\beta}_i = \beta_i + m-2$ for all $i$ with $\beta_i\neq 2$ and $\bar{\beta}_i =2$ if $\beta_i=2$.
\end{definition}

The operation $*$ can be thought of as starting with a permutation $\alpha$ that ends in 21, and inserting all of the elements greater than 2 from $\alpha$ to the front of $\beta$. 

\begin{example}\label{ex:star}
Let $\alpha = 34658721 = (1,3,6,7,2,4,5,8)$ and $\beta = 342561 = (1, 3, 2, 4, 5, 6)$. Then
\[ \alpha * \beta = \fbox{3}\ \fbox{4}\ \fbox{6} \ \fbox{5} \ \fbox{8} \ \fbox{7}\  9\ 10 \ 2 \ 11\ 12\ 1 = (1, 3, 6, 7, 9, 2, 4, 5, 8, 10, 11, 12).\]
A graphical depiction of this example can be found in Figure~\ref{fig:312 example}.
\end{example}

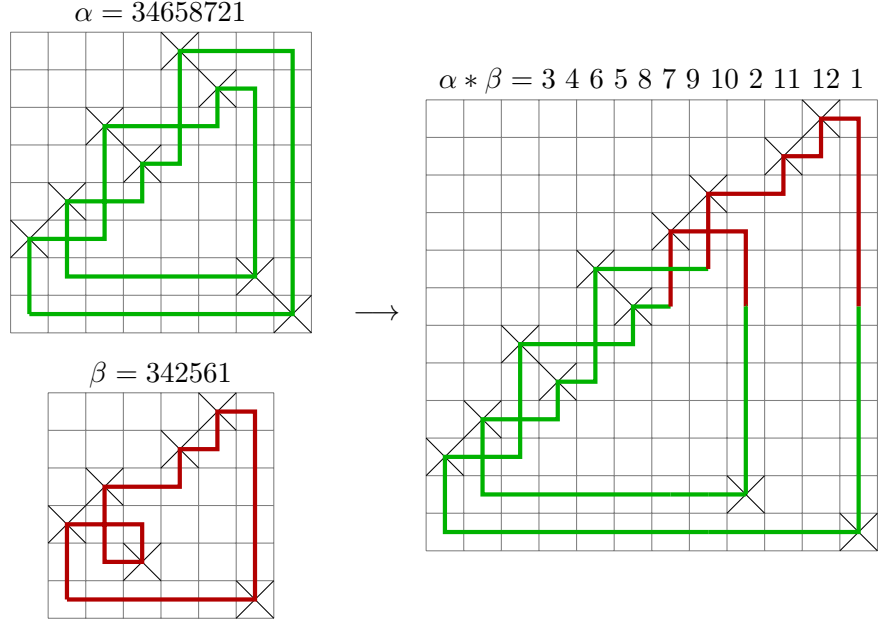
\begin{figure}[ht]
\centering
\begin{tabular}{c}
$\alpha = 34658721$ \\
\begin{tabular}{c}
\begin{tikzpicture}[scale=.5]
\newcommand\myx[1][(0,0)]{\pic at #1 {myx};}
\tikzset{myx/.pic = {\draw 
(-2.5mm,-2.5mm) -- (2.5mm,2.5mm)
(-2.5mm,2.5mm) -- (2.5mm,-2.5mm);}}
\draw[gray] (0,0) grid (8,8);
\foreach \x/\y in {
1/3,
2/4,
3/6,
4/5,
5/8,
6/7,
7/2,
8/1
}
\myx[(\x-.5,\y-.5)];
\draw[-, ultra thick,green!70!black] (.5,.5)--(.5,2.50000000000000)
--(2.50000000000000,2.50000000000000)--(2.50000000000000,5.50000000000000)
--(5.50000000000000,5.50000000000000)--(5.50000000000000,6.50000000000000)
--(6.50000000000000,6.50000000000000)--(6.50000000000000,1.50000000000000)
--(1.50000000000000,1.50000000000000)--(1.50000000000000,3.50000000000000)
--(3.50000000000000,3.50000000000000)--(3.50000000000000,4.50000000000000)
--(4.50000000000000,4.50000000000000)--(4.50000000000000,7.50000000000000)
--(7.50000000000000,7.50000000000000)--(7.50000000000000,0.500000000000000)
--(0.500000000000000,.5)--(.5,.5);
\end{tikzpicture} \end{tabular} \\[2cm]  $\beta = 342561$ \\ \begin{tabular}{c}
\begin{tikzpicture}[scale=.5]
\newcommand\myx[1][(0,0)]{\pic at #1 {myx};}
\tikzset{myx/.pic = {\draw 
(-2.5mm,-2.5mm) -- (2.5mm,2.5mm)
(-2.5mm,2.5mm) -- (2.5mm,-2.5mm);}}
\draw[gray] (0,0) grid (6,6);
\foreach \x/\y in {
1/3,
2/4,
3/2,
4/5,
5/6,
6/1
}
\myx[(\x-.5,\y-.5)];
\draw[-, ultra thick,red!70!black] (.5,.5)--(.5,2.5)
--(2.5,2.5)--(2.5,1.5)
--(1.5,1.5)--(1.5,3.5)
--(3.50000000000000,3.50000000000000)--(3.50000000000000,4.50000000000000)
--(4.50000000000000,4.50000000000000)--(4.50000000000000,5.50000000000000)
--(5.50000000000000,5.50000000000000)--(5.50000000000000,.50000000000000)
--(0.500000000000000,.5)--(.5,.5);
\end{tikzpicture}
\end{tabular}
\end{tabular}
$\longrightarrow$
\begin{tabular}{c}
$\alpha * \beta = 3 \ 4\  6\  5\  8\   7\   9\  10\ 2\ 11\ 12\ 1$ \\
\begin{tikzpicture}[scale=.5]
\newcommand\myx[1][(0,0)]{\pic at #1 {myx};}
\tikzset{myx/.pic = {\draw 
(-2.5mm,-2.5mm) -- (2.5mm,2.5mm)
(-2.5mm,2.5mm) -- (2.5mm,-2.5mm);}}
\draw[gray] (0,0) grid (12,12);
\foreach \x/\y in {
1/3,
2/4,
3/6,
4/5,
5/8,
6/7,
7/9,
8/10,
9/2,
10/11,
11/12,
12/1
}
\myx[(\x-.5,\y-.5)];
\draw[-, ultra thick,green!70!black] (7.5,.5)--(.5,.5)--(.5,2.50000000000000)
--(2.50000000000000,2.50000000000000)--(2.50000000000000,5.50000000000000)
--(5.50000000000000,5.50000000000000)--(5.50000000000000,6.50000000000000)
--(6.50000000000000,6.50000000000000);
\draw[-, ultra thick,red!70!black] (6.50000000000000,6.50000000000000) -- 
(6.50000000000000,8.50000000000000) -- (8.5, 8.5) -- (8.5, 6.5);
\draw[-, ultra thick,green!70!black] (8.5, 6.5) -- (8.5, 6.5);
\draw[-, ultra thick,green!70!black] (8.5, 6.5) -- (8.5, 1.5);
\draw[-, ultra thick,green!70!black] (7.5, 1.5) -- (6.5, 1.5);
\draw[-, ultra thick,green!70!black]  (7.5, 1.5) -- (8.5, 1.5);
\draw[-, ultra thick,green!70!black] (6.5, 1.5)--(1.5, 1.5)--(1.5, 3.5)--(3.5, 3.5) -- (3.5, 4.5) -- (4.5, 4.5) -- (4.5, 7.5) -- (7.5, 7.5);
\draw[-, ultra thick,red!70!black] (7.5, 7.5)--(7.5, 9.5) -- (9.5, 9.5) -- (9.5, 10.5) -- (10.5, 10.5) -- (10.5, 11.5) -- (11.5, 11.5) -- (11.5, 6.5);
\draw[-, ultra thick,green!70!black] (7.5,.5)--(11.5,.5)--(11.5,6.5)
;
\end{tikzpicture} 
\end{tabular}
\caption{This example demonstrates the operation defined in Definition~\ref{def:312-star}. Here, $\alpha$ ends in 21. Take all the elements before that 21 and append them to the front of $\beta$ to obtain $\alpha*\beta.$ The result is still cyclic and 312-avoiding.}
\label{fig:312 example}
\end{figure}

\begin{lemma}\label{lem:312oplus}
Let $m,n\geq 2$ and let 
 $\alpha \in \C_{m}(312)$ with $\alpha_{m-1}=2$ and $\beta \in \C_{n}(312)$. Then we have $\alpha * \beta \in \C_{n+m-2}(312)$.
\end{lemma}
\begin{proof}
    It is clear from the definition that the permutation $\pi = \alpha*\beta$ will be 312-avoiding, so we need only show it is cyclic.
    Note that by Lemma~\ref{lem:312Break2}, if $\beta_i=2$ for any $i\leq n-1$, then we can write $\beta=\beta_12\beta_21$, then if $\gamma21=\alpha*\beta_121$, then $\alpha*\beta=\gamma2\beta_21.$ Thus, by Lemma~\ref{lem:312Break2}, it is enough to show this lemma is true for $\beta\in\C_n$ with $\beta_{n-1}=2.$ 

    Now, since we have $\alpha=(1,a_2,\ldots,a_r, m-1,2,a_{r+2},\ldots,a_{n-1}, m)$ and $\beta=(1,b_2,\ldots, b_{s},n-1,2,b_{s+2},\ldots, b_{n-1},n)$, we will have 
    \[
    \alpha*\beta=(1,a_2,\ldots,a_r,m-1,\bar{b}_2,\ldots, \bar{b}_{s},n-1,2,a_{r+2},\ldots,a_{n-1}, m,\bar{b}_{s+2},\ldots, \bar{b}_{n-1},n)
    \]
    where $\bar{b}_j=b_j+ m-2$. This works because in $\pi=\alpha*\beta$, we have that $\pi_{m-1}=\beta_1$, $\pi_{m}=\beta_2$ and that the $\alpha$ values and positions are all the same except for $\alpha_{m-1}$ and $\alpha_m,$ while all the $\beta$ values and positions have been shifted by $m-2$, except for $\beta_{n-1}$ and $\beta_n$.
\end{proof}

Note that $\pi=3421 * \beta$ will always be a permutation that starts with $\pi_1\pi_2=34$. Let us define
\[ X_{2,n} = \{3421 * \beta : \beta \in C_{n-2}(312)\}.\] It is clear that $|X_{2,n}| = c_{n-2}$ and that $X_{1,n} \cap X_{2,n} = \emptyset.$ In order to build all permutations in $C_n(312),$ we now state a  lemma generalizing this idea, which will allow us to find a recursive lower bound on $c_n(312).$

\begin{lemma}\label{lem:312-Ak} Let
\[ X_{1,n} = \{ 2\beta: \beta \in \C_{n-1}(312)\}.\]
For $k \geq 2,$ let
\[ A_k = \{\alpha \in \C_{k+2}(312): \alpha_{k+1} = 2, \alpha \notin X_{1, k+2} \cup X_{2,k+2} \cup \cdots \cup X_{k+1,k+2}\}\]
and let
 \[ X_{k,n} = \{\alpha * \beta: \alpha \in A_k, \beta \in \C_{n-k}(312)\}.\]
 Then, $|X_{k,n}| = |A_k|c_{n-k}(312)$ for $k \geq 2.$
\end{lemma}

\begin{proof}
This is clear since we can undo the $*$ operator in this case and obtain one permutation in $A_k$ and one in $\C_{n-k}(312)$. Indeed, for any $\pi=\alpha*\beta\in X_{k,n},$  we know $\alpha=\pi_1\ldots\pi_k21$ and $\beta = \red(\pi_{k+1}\ldots \pi_n)$.
\end{proof}

The sets $A_k$ for the first few values of $k$ are \[A_2 = \{3421\}, \quad A_3 = \{43521, 35421\}, \quad A_4 = \{436521, 546321, 465321\}.\]  Notice $345621 \notin A_4$ even though $345621 \in \C_6(312)$ because  $3421 \in A_2$, and thus $345621\in X_{2,6}$ since $345621=3421*3421$. Since by construction the sets $|X_{k,n}|$ have no intersection, the sum of the sizes of these sets give a lower bound on $c_n(312)$. Using just the sets $X_{1, n}, X_{2,n}, X_{3, n},$ and $X_{4,n}$, we have
\[ c_n(312) \geq \left|\bigcup_{i=1}^{4} X_{i,n} \right|  = c_{n-1}(312) + c_{n-2}(312) + 2c_{n-3}(312) + 3c_{n-4}(312),\] which implies $c_n(312)$ has a growth rate of at least 2.1746. Using a computer algebra system \cite{Sage}, we can compute values of $|A_k|$   to achieve a growth rate for $c_n(312)$ of slightly more than 3. The first 18 values of $|A_k|$ are found in Table~\ref{312ATable}.

\begin{table}
\centering
%
\begin{tabular}{c|p{1.5cm}||c|p{1.5cm}||c|p{1.5cm}}
    $k$ & $|A_k|$ &  $k$ & $|A_k|$ &  $k$ & $|A_k|$ \\ \hline\hline
    1 & 1 & 7 & 64 & 13 & 59,932 \\ 
    2 & 1 & 8 & 175 & 14 & 201,544 \\
    3& 2 & 9 & 526 & 15 & 685,550 \\
    4 & 3 & 10 & 1726 & 16 & 2,350,957 \\
    5 & 6 & 11 & 5612 & 17 & 8,108,124 \\ 
    6 & 22 & 12 & 18,231 & 18 & 28,174,500
\end{tabular}
\caption{Values of $|A_k|$, as defined in Lemma~\ref{lem:312-Ak}, for $k \in [18]$. These are exactly those permutations in $\C_{k+2}(312)$ that end with $21$ and cannot be obtained by $\alpha*\beta$ for some smaller $\alpha$ and $\beta$.} \label{312ATable}
\end{table}




\begin{theorem} \label{thm:312} Let $n \geq 19$ and let $c_n = |\C_n(312)|.$ Then
\begin{align*}
    c_n & \geq c_{n-1} + c_{n-2} + 2c_{n-3} + 3c_{n-4}  +6c_{n-5} + 22c_{n-6} + 64c_{n-7} + 175c_{n-8} + 526c_{n-9} \\
& \ \ \ \ + 1726c_{n-10}+5612c_{n-11} + 18231c_{n-12}+59932c_{n-13} +201544c_{n-14}\\
& \ \ \ \ +685550c_{n-15} + 2350957c_{n-16} + 8108124c_{n-17} + 28174500c_{n-18}
\end{align*}
which implies that $c_n(312)$ has a growth rate greater than 3.04385969.
\end{theorem}

\begin{proof} We know that $c_n \geq \sum_{k=1}^{18}|X_{k,n}|$ since the sets $X_{k,n}$ are disjoint by construction. Furthermore, we have that $|X_{k,n}| = |A_{k}|c_{n-k}$ by Lemma~\ref{lem:312-Ak}. By computing the first 18 values of $|A_k|$ we have our result.
\end{proof}

We end this section with a short discussion about other observations. 
First, let us note that by doing this $*$ operation, we are always adding ``small'' numbers to the beginning of the permutation. By considering the reverse-complement-inverse, we could also add ``large'' numbers to the end of a permutation (right before the 1). For example, we saw that we can always insert a 2 in the first position and remain cyclic and 312-avoiding. However, we could also always add an $n$ in the second-to-last position. There is a lot of intersection in these cases, so it does not improve the growth rate. However, by taking the results from Theorem~\ref{thm:312}, and combining it with this insertion of $n$ in the second-to-last position, you get 
\begin{align*}
    c_n  &\geq 2c_{n-1} + c_{n-3} + c_{n-4}  +3c_{n-5} + 16c_{n-6} + 42c_{n-7} + 111c_{n-8} + 351c_{n-9} \\
& \ \ \ \ + 1200c_{n-10}+3886c_{n-11} + 12619c_{n-12}+41701c_{n-13} +141612c_{n-14}\\
& \ \ \ \ +484006c_{n-15} + 1665407c_{n-16} + 5757167c_{n-17} + 20066367c_{n-18}-28174500c_{n-19},
\end{align*}
but since $c_{n-18}\geq c_{n-19}$ and $c_{n-17}\geq c_{n-18}+c_{n-19}\geq 2c_{n-19}$, we can actually replace those last three terms with $+3406201c_{n-19}$. In particular, this proves a stronger verrsion of B\'{o}na and Cory's conjecture \cite{BC2019} that $c_n(312)\geq 2c_{n-1}(312).$

Next, let us note that as a consequence of Lemma \ref{lem:312Break2}, cyclic 312-avoiding permutations can be counted in terms of the number of cyclic 312-avoiding permutations that end in 21.

\begin{corollary} Let $n \geq 2$ Then \[ c_n(312) = \sum_{k=2}^{n} a_k(312)c_{n+1-k}(312) \]
where $a_k(312) =  |\{\pi \in \C_k(312) : \pi_{k-1}\pi_k = 21\}|$.
\end{corollary}

Therefore, if we could count the number of cyclic 312-avoiding permutations that end in 21, we could enumerate all cyclic 312-avoiding permutations. 
This idea can actually be extended further. Indeed, if we let $a'_n(312)$ be the number of permutations $\pi\in\C_n(312)$ that end in $\pi_{n-1}\pi_n=21$ and begin with $\pi_1=3,$ and let $b_n(312)$ be the number of permutations $\pi\in\C_n(312)$ that end with $\pi_{n-2}\pi_{n-1}\pi_n=321$, and $b'_n(312)$ be the number that end in $\pi_{n-3}\pi_{n-2}\pi_{n-1}\pi_n=3n21$,
then \[a_n(312)= \sum a'_{n-j}(312)b'_{j+4} + \sum a_{n-j-1}b_{j+3}\]
for reasons similar to those in the proof of Lemma~\ref{lem:312Break2}. In short, if $\pi =\gamma3\alpha21\in\C_{n}(312),$ then either $\gamma3n21$ and $3\alpha21$ are both cyclic and 312-avoiding or $\gamma321$ and $\alpha21$ are both cyclic and 312-avoiding. 

Though there does seem to be some nice structure to these permutations, counting them exactly still appears to be rather difficult question.

\section{321-avoiding cycles}

In this section, we work with permutations in $\C_n(321).$ Similar to the last section, we define a binary operation $\odot$ that lets us augments a 321-avoiding cyclic permutation by another 321-avoiding cyclic permutation to build larger ones, though this one is slightly more complicated.

\begin{definition}\label{defn:321-combine}
Let $\alpha = \alpha_1\alpha_2\cdots\alpha_m \in \S_m$ and $\beta = \beta_1\beta_2\cdots\beta_n \in \S_n$ for $m,n\geq 3$. Further suppose $\alpha_r = m-1$ and $\alpha_s = m$ for some $r < s < m$. Define the permutation $\alpha \odot \beta\in\S_{m+n-3}$ by
\[ \alpha \odot \beta = \alpha_1\alpha_2\cdots\alpha_{r-1}\bar{\beta_1}\alpha_{r+1}\alpha_{r+2}\cdots\alpha_{s-1}\bar{\beta_2}\alpha_{s+1}\alpha_{s+2}\cdots\alpha_{m-1}\bar{\beta_3}\bar{\beta_4}\cdots\bar{\beta_{n}} \]
where
\[ \bar{\beta_i} = \begin{cases} \alpha_m & \text{if } \beta_i=1 \\
\beta_i+m-3 & \text{otherwise.} \end{cases}\]
\end{definition}

In essence, the permutation $\alpha \odot \beta$ can be thought of as ``inserting" elements into $\beta$, or ``augmenting" $\beta$ in the following way. It is defined when $m-1$ appears before $m$ in $\alpha$ in positions $r$ and $s$ respectively, and is formed by starting with $\beta$ and
\begin{itemize}
\item inserting $\alpha_{s+1}\cdots\alpha_{m-1}$ into the third position of $\beta$ (and shifting elements in $\beta$ appropriately),
\item inserting $\alpha_{r+1}\cdots\alpha_{s-1}$ into the second position of $\beta$ (and shifting elements in $\beta$ appropriately), and
\item inserting $\alpha_{1}\cdots\alpha_{r-1}$ into the first position of $\beta$ (and shifting elements in $\beta$ appropriately).
\end{itemize}

We will show that under certain conditions, the product $\alpha\odot \beta$ of 321-avoiding cyclic permutations is itself a 321-avoiding cyclic permutation. The following example demonstrates how this definition works in practice, particularly in the cases where both $\alpha$ and $\beta$ are cyclic 321-avoiding permutations.

\begin{example}
    Suppose $\alpha = 3561724 = (1,3,6,2,5,7,4)$ and $\beta=34568127 =(1, 3, 5, 8, 7, 2, 4, 6)$. In $\alpha$, we have that the value  $m-1=6$ is in position $r=3$ and the value $m=7$ is in position $s=5.$ Then,
    \[
    \alpha\odot\beta = \fbox{3} \ \fbox{5} \ \mathbf{7} \ \fbox{1} \ \mathbf{8} \ \fbox{2} \ \mathbf{9 \ 10 \ 12 \ 4 \ 6 \ 11} = (1,3,7,9,12,11,6,2,5,8,10,4).
    \]
    The bold elements correspond to $\beta$ and the boxed elements correspond to the elements from $\alpha$ that have been inserted into $\beta$. This example is illustrated in Figure~\ref{fig:321 example}.
\end{example}

\begin{figure}[ht]
\centering
\begin{tabular}{c}
$\alpha = 3561724$ \\
\begin{tabular}{c}
\begin{tikzpicture}[scale=.5]
\newcommand\myx[1][(0,0)]{\pic at #1 {myx};}
\tikzset{myx/.pic = {\draw 
(-2.5mm,-2.5mm) -- (2.5mm,2.5mm)
(-2.5mm,2.5mm) -- (2.5mm,-2.5mm);}}
\draw[gray] (0,0) grid (7,7);
\foreach \x/\y in {
1/3,
2/5,
3/6,
4/1,
5/7,
6/2,
7/4
}
\myx[(\x-.5,\y-.5)];
\draw[-, ultra thick,green!70!black] (.5,.5)--(.5,2.5)
--(2.5,2.5)--(2.5,5.5)--(5.5,5.5)--(5.5,1.5)--(1.5,1.5)--(1.5,4.5)--(4.5,4.5)--(4.5,6.5)--(6.5,6.5)--(6.5,3.5)--(3.5,3.5)--(3.5,.5)--(.5,.5);
\end{tikzpicture} \end{tabular} \\[2cm]  $\beta=34568127$ \\ \begin{tabular}{c}
\begin{tikzpicture}[scale=.5]
\newcommand\myx[1][(0,0)]{\pic at #1 {myx};}
\tikzset{myx/.pic = {\draw 
(-2.5mm,-2.5mm) -- (2.5mm,2.5mm)
(-2.5mm,2.5mm) -- (2.5mm,-2.5mm);}}
\draw[gray] (0,0) grid (8,8);
\foreach \x/\y in {
1/3,
2/4,
3/5,
4/6,
5/8,
6/1,
7/2,
8/7
}
\myx[(\x-.5,\y-.5)];
\draw[-, ultra thick,red!70!black] (.5,.5)--(.5,2.5)
--(2.5,2.5)--(2.5,4.5)
--(4.5,4.5)--(4.5,7.5)--(7.5,7.5)--(7.5,6.5)--(6.5,6.5)--(6.5,1.5)--(1.5,1.5)--(1.5,3.5)--(3.5,3.5)--(3.5,5.5)--(5.5,5.5)--(5.5,.5)--(.5,.5);
\end{tikzpicture}
\end{tabular}
\end{tabular}
$\longrightarrow$
\begin{tabular}{c}
$\alpha\odot\beta = 3 \ 5 \ 7 \ 1 \ 8 \ 2 \ 9 \ 10 \ 12 \ 4 \ 6 \ 11$ \\
\begin{tikzpicture}[scale=.5]
\newcommand\myx[1][(0,0)]{\pic at #1 {myx};}
\tikzset{myx/.pic = {\draw 
(-2.5mm,-2.5mm) -- (2.5mm,2.5mm)
(-2.5mm,2.5mm) -- (2.5mm,-2.5mm);}}
\draw[gray] (0,0) grid (12,12);
\foreach \x/\y in {
1/3,
2/5,
3/7,
4/1,
5/8,
6/2,
7/9,
8/10,
9/12,
10/4,
11/6,
12/11
}
\myx[(\x-.5,\y-.5)];
\draw[-, ultra thick,green!70!black] (3.5,3.5)--(3.5,.5)--(.5,.5)--(.5,2.5)--(2.5,2.5)--(2.5,5.5);
\draw[-, ultra thick,red!70!black](2.5,5.5) --(2.5,6.5)--(6.5,6.5)--(6.5,8.5)--(8.5,8.5)--(8.5,11.5)--(11.5,11.5)--(11.5,10.5)--(10.5,10.5)--(10.5,5.5)--(5.5,5.5);
\draw[-, ultra thick,green!70!black](5.5,5.5)--(5.5,1.5)--(1.5,1.5)--(1.5,4.5)--(4.5,4.5)--(4.5,5.5);
\draw[-, ultra thick,red!70!black] (4.5,5.5)--(4.5,7.5)--(7.5,7.5)--(7.5,9.5)--(9.5,9.5)--(9.5,3.5)--(3.5,3.5);
\end{tikzpicture} 
\end{tabular}
\caption{This example demonstrates the operation defined in Definition~\ref{defn:321-combine}. Here $\alpha\in\C_7(321)$ and $\beta\in\C_8(321)$. In this case, by Lemma~\ref{lem:321-all}, $\alpha\odot\beta$ is still cyclic and 321-avoiding since $\alpha_7\neq 6$ and $\beta_2\neq 1$.}
\label{fig:321 example}
\end{figure}
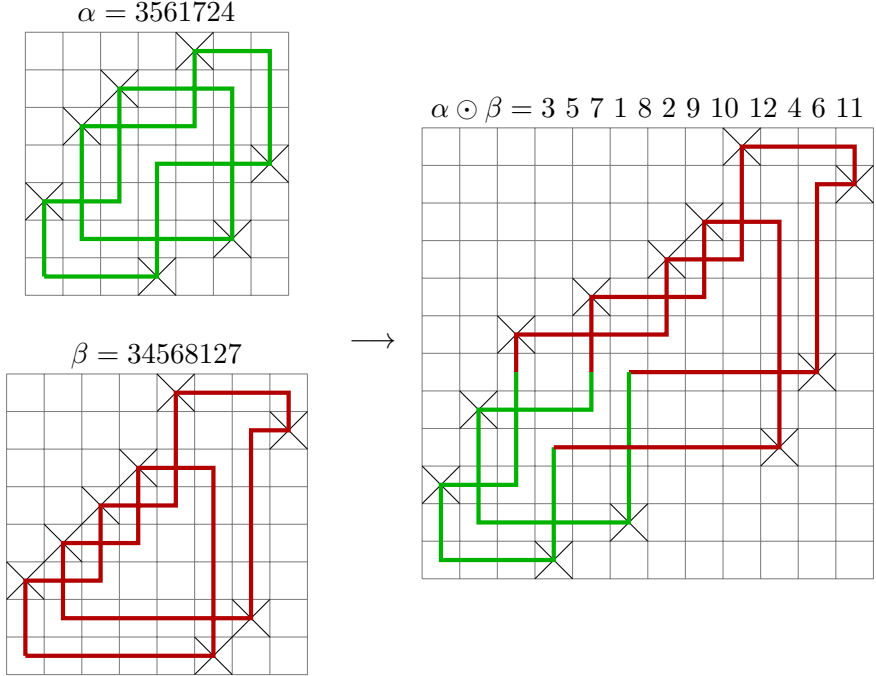

Let us first show that under certain circumstances, given cyclic 321-avoiding permutations $\alpha$ and $\beta,$ the product $\alpha\odot\beta$ is also a cyclic 321-avoiding permutations.
Note that if $k = 0$, then $\alpha \odot \beta = \beta$, so we consider the case when $k\geq 1$ in the next lemma.

\begin{lemma} \label{lem:321-all} Suppose $k \geq 1$ and $n \geq 3$, and let $\alpha = \alpha_1\alpha_2\cdots\alpha_{k+3}  \in \C_{k+3}(321)$ and $\beta = \beta_1\beta_2\cdots\beta_n \in \C_{n}(321)$. 
If $k=1$, then $\alpha \odot \beta \in \C_{n+1}(321)$. For $k\geq 2$, if both $\alpha_{k+3} \neq k+2$ and $\beta_2 \neq 1,$ then $\alpha \odot \beta \in \C_{n+k}(321)$.
\end{lemma}

\begin{proof}  If $k = 1$, and there are only two possible cyclic 321-avoiding permutations for $\alpha$, namely 3142 and 2341, that satisfy the condition that 3 appears before 4 in the one-line notation. By definition $3142 \odot \beta = \bar{\beta}_11\bar{\beta}_2\bar{\beta}_3\cdots\bar{\beta}_{n-1}$
where $\bar{\beta}_i=\beta_i +1$ for all $i$. It is clear that $3142 \odot \beta$ is 321-avoiding; the permutation is also cyclic with its cycle structure obtained by inserting the element 2 at the end of the cycle and shifting the other elements, excepting 1, by one. Similarly, $2341 \odot \beta$ is formed by inserting the element 2 into the first position of $\beta$ and is also cyclic 321-avoiding.


For $k \geq 2$, since $\alpha_{k+3} \neq k+2$, there exist $1 \leq r < s < k+3$ so that $\alpha_r = k+2$ and $\alpha_s = k+3$. Indeed, we know that $r < s$ because otherwise $\alpha_s\alpha_r\alpha_{k+3}$ would be a 321 pattern. Now consider
\[ \alpha \odot \beta = \alpha_1\alpha_2\cdots\alpha_{r-1}\bar{\beta}_1\alpha_{r+1}\alpha_{r+2}\cdots\alpha_{s-1}\bar{\beta}_2\alpha_{s+1}\alpha_{s+2}\cdots\alpha_{k+2}\bar{\beta}_3\bar{\beta}_4\cdots\bar{\beta}_{n} \]
where
\[ \bar{\beta}_i = \begin{cases} \alpha_{k+3} & \text{if } \beta_i=1 \\
\beta_i+k & \text{otherwise.} \end{cases}\]

Because all of the $\alpha_i$'s and all of the $\bar{\beta}_j$'s  in $\alpha \odot \beta$ appear in the same relative order as they did in the permutations $\alpha$ and $\beta$ respectively, any 321-pattern in $\alpha \odot \beta$ must include at least one $\alpha_i$ and at least one $\bar{\beta}_j$. Furthermore, since $\beta_1 \neq 1$ (since $\beta$ is cyclic) and $\beta_2 \neq 1$ (by stated hypotheses), $\bar{\beta_1} > \alpha_i$ and $\bar{\beta_2} > \alpha_i$ for all $i$, and $\bar{\beta_1} < \bar{\beta_2}$ since otherwise $\beta_1\beta_21$ would be a 321 pattern in $\beta$.

Suppose toward a contradiction that $\bar{\beta_i}\alpha_j\alpha_m$ is a 321-pattern in $\alpha \odot \beta$ for $i \in \{1, 2\}$ and some $j < m$. Then either $\alpha_r\alpha_j\alpha_m$ (if $i=1$) or $\alpha_s\alpha_j\alpha_m$ (if $i=2$) is a 321 pattern in $\alpha$ which is a contradiction.  Now suppose $\bar{\beta_i}\alpha_j\bar{\beta_m}$ is a 321 pattern in  $\alpha \odot \beta$ for $i \in \{1,2\}$ and $m > 2$. Since $\alpha_j > \bar{\beta_m},$ we must have $\bar{\beta_m} = \alpha_{k+3}$ and then either $\alpha_r\alpha_j\alpha_{k+3}$ (if $i=1$) or $\alpha_s\alpha_j\alpha_{k+3}$ (if $i=2$) is a 321-pattern in $\alpha$ giving another contradiction. Thus $\alpha \odot \beta$ is 321-avoiding.


To see that $\alpha \odot \beta$ is cyclic, we begin by considering the cycle notation of $\alpha$ and $\beta$. First suppose that $k+3$ comes before $k+2$ in the cycle notation of $\alpha$ and that $k+3$ occurs in position $i$ of the cycle and $k+2$ occurs at position $j > i$. Thus, we can write
\[ \alpha = (1, a_2, \ldots, a_{i-2}, s, k+3, a_{i+1}, \ldots, a_{j-2}, r, k+2, a_{j+1}, \ldots, a_{k+3}).\]
For the cycle notation of $\beta$, suppose 2 is in position $m<n$ and write
\[ \beta = (1, b_2, \ldots, b_{m-1}, 2, b_{m+1}, \ldots, b_n).\] Then the cycle notation of $\alpha \odot \beta$ is given by
\[ \alpha \odot \beta = (1,  a_2, \ldots, a_{i-2}, s, \overline{b_{m+1}}, \ldots, \overline{b_n}, a_{i+1}, \ldots, a_{j-2}, r, \overline{b_2}, \ldots, \overline{b_{m-1}}, k+2, a_{j+1}, \ldots, a_{k+3}) \]
where $\overline{b_i} = b_i+k.$ 

The case where $k+3$ comes after $k+2$ in the cycle of $\alpha$ is similar. In this case, if 
\[ \alpha = (1, a_2, \ldots, a_{i-2}, r, k+2, a_{i+1}, \ldots, a_{j-2},s, k+3, a_{j+1}, \ldots, a_n),\] 
then
\[ \alpha \odot \beta = (1,  a_2, \ldots, a_{i-2}, r, \overline{b_{2}}, \ldots, \overline{b_{m-1}}, k+2, a_{i+1}, \ldots, a_{j-2}, s, \overline{b_{m+1}}, \ldots, \overline{b_n}, a_{j+1}, \ldots, a_{k+3}). \]
\end{proof}



Note that if $\gamma \in \C_{n}(321)$ with $n\geq 3$, then it cannot be the case that both $\gamma_1=2$ and $\gamma_2=1$ since $\gamma$ is cyclic. If $\gamma_1=2,$ then deleting the element 2 from $\gamma$ results in a cyclic 321-avoiding permutation, and if $\gamma_2=1,$ deleting the element 1 results in a cyclic 321-avoiding permutation. Therefore, in the case when $k=1$ in the previous lemma, the number of permutations in $\C_n(321)$ of the form $\alpha\odot\beta$ is equal to $2c_{n-1}(321)$, and thus we attain a bound of $c_{n}(321) \geq 2c_{n-1}(321).$ However, this growth rate can be improved by considering $\alpha$ of length $m=k+3$ with $m \geq 5.$ We can easily find that $c_5(321) = 10,$ and exactly 2 of these permutations, namely 34512 and 45123, do not start with 2 or have a 1 in the second position. Consider the permutations $34512 \odot \beta$ and $45123 \odot \beta$ whose forms are
\begin{equation}\label{eqn:321-34512}
34512 \odot \beta = \fbox{3}\bar{\beta_1}\bar{\beta_2}\fbox{1}\bar{\beta_3}\bar{\beta_4}\cdots\bar{\beta_n}\end{equation}
and
\begin{equation}\label{eqn:321-45123} 45123 \odot \beta = \bar{\beta_1}\bar{\beta_2}\fbox{1}\fbox{2}\bar{\beta_3}\bar{\beta_4}\cdots\bar{\beta_n}.\end{equation}

Note that in order for the resulting augmented permutations to be $321$-avoiding, it is necessary that $\beta_1 < \beta_2.$ If $\beta$ is cyclic 321-avoiding, this condition is equivalent to $\beta$ not having a 1 in the second position.

Furthermore, because the inverse of any cyclic 321-avoiding permutation is also cyclic 321-avoiding, the permutations $(34512 \odot \beta)^i$ and $(45123 \odot \beta)^i$ are possible ways to build larger cyclic 321-avoiding permutations. These operations correspond to taking a permutation $\beta'$ with $\beta_1'\neq 2$, and in the first case, inserting 4 at the beginning of $\beta'$ and a 1 after $\beta'_1$, and in the second case, appending 34 to the beginning of $\beta'$, shifting the remaining elements of $\beta'$ appropriately.
 Unfortunately, this additional method of augmenting permutations produces some permutations that could be achieved without taking inverses, so some care must be taken in order to avoid double-counting. The next lemma below provides the generalized result.
In this lemma, we define sets $X_{k,n}$ to be cyclic 321-avoiding permutations of the form $\alpha\odot\beta$ or $(\alpha \odot \beta)^i$ for some $\alpha$ in the set $A_k$, which is defined to be those elements of $\C_{k+3}(321)$ that are not already of the form $\alpha'\odot\beta'$ or $(\alpha'\odot\beta')^i$ for some smaller $\alpha'.$

\begin{lemma}\label{lem:321-Xk} Let $k \geq 1$ and $n\geq k+3$. Let
\[ X_{1,n} = \{\pi \in \C_n(321):  \pi_1 = 2 \text{ or } \pi_2=1\}, \]
and for $k\geq 2,$ let
\begin{align*} X_{k,n} &= \{\alpha \odot \beta, (\alpha \odot \beta)^i: \alpha \in A_k, \beta \in \C_{n-k}(321) \text{ with } \beta_2 \neq 1\} 
\end{align*}
where 
\[A_k = \{\alpha \in \C_{k+3}: \alpha \notin X_{1,k+3}\cup X_{2, k+3}\cup \cdots \cup X_{k-1, k+3}\}.\]
Then $|X_{1,n}| = 2c_{n-1}(321)$ and $|X_{k,n}| = |A_k|(2c_{n-k} - 3c_{n-k-1})$ for $k \geq 2.$ 
\end{lemma}

\begin{proof} 
As described above, it is clear that $|X_{1,n}| = 2c_{n-1}(321)$, it remains to consider the case where $k \geq 2$. It is easy to see that if
\[ Y = \{\alpha \odot \beta: \alpha \in A_k, \beta \in \C_{n-k}(321) \text{ with } \beta_2 \neq 1\}, \]
then $|Y| = |A_k|(c_{n-k}(321) - c_{n-k-1}(321))$ because the number of permutations $\beta\in\C_{n-k}(321)$ with $\beta_2=1$ is equal to $c_{n-k-1}$. Similarly, if
\[ Z = \{(\alpha \odot \beta)^i: \alpha \in A_k, \beta \in \C_{n-k}(321) \text{ with } \beta_2 \neq 1\}, \]
then $|Z| = |Y|,$ and we need only to count $|Y \cap Z|.$ To that end, suppose $\pi \in Y \cap Z.$ Then we must have $\pi = \alpha \odot \beta$ and $\pi = (\hat{\alpha} \odot \hat{\beta})^i$ for some $\alpha, \hat{\alpha} \in A_k\subseteq \S_{k+3}$ and $\beta, \hat{\beta} \in \C_{n-k}$ with $\beta_2 \neq 1$ and $\hat{\beta_2} \neq 1$. Let us note that by the definition of $\alpha\odot\beta$, the first $k+2$ elements of $\pi$ must contain at least $k$ elements less than or equal to $k+2$ (namely those that come from $\alpha$). Notice that a similar argument is true by definition of $(\hat\alpha\odot \hat\beta)^i$ for those that come from $\hat\alpha^i$. However, $(\hat\alpha\odot \hat\beta)^i$ is obtained by combining $\hat\alpha^i$ and $\hat\beta^i$ as:
\[(\hat\alpha\odot \hat\beta)^i=\hat\alpha^i_1\ldots\hat\alpha^i_{R-1}\hat\beta^i_1\hat\alpha^i_{R+1}\ldots\hat\alpha^i_{k+1}\hat\beta^i_2\hat\beta^i_3\ldots\hat\beta^i_{n-k}
\]
where $\hat\alpha^i_R=k+3$ and the values of $\hat\beta^i$ have been shifted appropriately. Thus if there are exactly $k$ elements less than or equal to $k+2$, (that is, if all such elements come from $\alpha$ and $\hat\alpha^i$ only), then $\alpha_{k+2}=k+3$. But then $\alpha=\alpha'\odot 2341$ where $\alpha'$ is obtained by deleting $k+3,$ and thus $\alpha\not\in A_k.$

Thus, since $\pi$ is cyclic, this implies that $k-2$ of the first $k-1$ elements of $\pi$ are less than or equal to $k-1.$ However, these include $\bar\beta_1$ and $\bar\beta_2$.  If $\beta_1$ and $\beta_2$ were both greater than 2, then $\bar\beta_1$ and $\bar\beta_2$ would both be greater than $k-1$. Since $\beta_2\neq 1,$ we must therefore have $\beta_1=2.$ Finally, since $\pi\in Y\cap Z$ exactly when $\beta=2$ and there are $c_{n-k-1}$ such $\beta$ of length $n-k$, we have that $|Y\cap Z|=|A_k|c_{n-k-1}$. The result follows.
\end{proof}

\begin{remark}
    It is clear from the proof that the converse of Lemma~\ref{lem:321-Xk} is also true. In particular,  if $\gamma \in \C_n(321)$ and there exists $\alpha \in A_k$ with $\alpha_r=k+2, \alpha_s=k+3$ and either $\gamma_j = \alpha_j$ for all $j \in [1,k-1] \setminus \{r,s\}$ or $(\gamma^i)_j = \alpha_j$ for all $j \in [1,k-1] \setminus \{r,s\}$, then $\gamma \in X_{k,n}$.
\end{remark}


By construction, the  sets $X_{k,n}$ are disjoint. As described above, it is clear that $A_1=\{3142,2341\}$ and $A_2=\{34512,45123\}.$ Using Lemma \ref{lem:321-Xk}, this gives us the the bound \[ c_n(321) \geq 2c_{n-1}(321) + 4c_{n-2}(321) - 6c_{n-3}(321),\] which yields a growth rate of approximately 2.6554. However, we can continue to compute the size of $A_k$ for some $k>2,$ which improves this result.  

Lemma \ref{lem:321-Xk} implies that \[ |A_k| = c_{k+3} - |X_{1,k+3}| - |X_{2, k+3}| - \cdots- |X_{k-1,k+3}|.\] Thus, using the fact that $|X_{2,n}| = 4c_{n-1}(321) - 6c_{n-3}(321)$, we can inductively apply Lemma~\ref{lem:321-Xk} to see that:
\[|A_3| = c_6 - |X_{1,6}| - |X_{2,6}| = c_6-(2c_5) - (4c_4 - 6c_3) = 0 \]
and
\[ |X_{3,n}|=|A_3|(2c_{n-3}-3c_{n-4}) = 0.\] However, for $k=4$, we have $|X_{1,7} \cup X_{2,7}| = 64$, while $c_7(321) = 66$; in other words, we miss 2 permutations in $\C_7(321)$ if we only use the $\odot$ operation.  These two missed permutations are exactly those in $A_4=\{5671234,4567123\},$ so $|A_4| = 2.$ Thus, for any $n\geq 7$, we have $|X_{4,n}| = 2(2c_{n-4} - 3c_{n-4}).$ This gives the bound
\[ c_n(321) \geq 2c_{n-1}(321) + 4c_{n-2}(321) - 6c_{n-3}(321) + 4c_{n-4}(321) - 6c_{n-5}(321)\] which yields a growth rate of approximately 2.7486. 
We can continue this process indefinitely as long as we know the value of $c_k(321)$.  Table~\ref{fig:321-1} summarizes the values of $|X_{k,n}|$ and $|A_k|$  for $1\leq k \leq 16$ that will be used in the next theorem.

\begin{table}[ht]
\begin{center}
\begin{tabular}{l|l|l|l|l}
$k$ & $c_{k+3}$ & $\displaystyle{\sum_{i=1}^{k-1}|X_{i,k+3}|}$ & $|A_k|$ & $|X_{k,n}|$\\ \hline
1 & 4 & & & $2c_{n-1}$\\
2 & 10 & 8 & 2 & $4c_{n-2}-6c_{n-3}$\\
3 & 24 & 24 & 0 & 0\\
4 & 66 & 64 & 2 & $4c_{n- 4 } -  6 c_{n- 5 }$\\
5 & 178 & 172 & 6 & $12c_{n- 5 } -  18 c_{n- 6 }$\\
6 & 512 & 504 & 8 & $16c_{n- 6 } -  24 c_{n- 7 }$\\
7 & 1486 & 1440 & 46 & $92c_{n- 7 } -  138 c_{n- 8 }$\\
8 & 4446 & 4336 & 110 & $220c_{n- 8 } -  330 c_{n- 9 }$\\
9 & 13468 & 13172 & 296 & $592c_{n- 9 } -  888 c_{n- 10 }$\\
10 & 41648 & 40512 & 1136 & $2272c_{n- 10 } -  3408 c_{n- 11 }$\\
11 & 130178 & 126948 & 3230 & $6460c_{n- 11 } -  9690 c_{n- 12 }$\\
12 & 412670 & 402288 & 10382 & $20764c_{n- 12 } -  31146 c_{n- 13 }$\\
13 & 1321418 & 1286184 & 35234 & $70468c_{n- 13 } -  105702 c_{n- 14 }$\\
14 & 4274970 & 4161840 & 113130 & $226260c_{n- 14 } -  339390 c_{n- 15}$\\
15 & 13948966 & 13571812 & 377154 & $754308c_{n- 15 } -   1131462c_{n- 16}$\\
16 & 45890440 & 44616288 & 1274152 & $2548304c_{n- 16 } -   3822456c_{n- 17}$
\end{tabular}
\end{center}
\caption{For each $k$, we provide the value of $c_{k+3}:=c_{k+3}(321)$, the number of permutations generated by $\pi=\alpha\odot\beta$ for $\alpha\in A_i$ with $i<k$, the number that are not generated of this form (i.e., those in the set $A_k$), and the total number of permutations in $\C_n(321)$ of the form $\pi=\alpha\odot\beta$ for some $\alpha\in A_k$.}
\label{fig:321-1}
\end{table}

\begin{theorem}\label{thm:321}
Let $n \geq 19$ and let $c_n = |\C_n(321)|.$ Then
\begin{align*}
    c_n \geq& \ 2c_{n-1}+4c_{n-2}-6c_{n-3} + 4c_{n-4}+6c_{n-5}-2c_{n-6}+68c_{n-7} +82c_{n-8} +262c_{n-9}\\
    &+1384c_{n-10}+3052c_{n-11}+11074c_{n-12}+39322c_{n-13}+120558c_{n-14}+414918c_{n-15}\\
    &+1416842c_{n-16} - 3822456c_{n-17}
\end{align*}
which implies that $c_n$ growth rate of approximately 3.178858.
\end{theorem}

\begin{proof} We know that for $n\geq 19$, $c_n \geq \sum_{k=1}^{16} |X_{k,n}|$. Furthermore, we have
\begin{align*} \displaystyle{\left|\bigcup_{k=1}^{16} X_{k,n}\right|} &= 2c_{n-1} + \sum_{k=2}^{16} |A_k| (2c_{n-k} - 3c_{n-k}) \\ &= 2c_{n-1} + 2|A_2|c_{n-2} + \left(\sum_{k=3}^{16} (2|A_k| - 3|A_{k-1}|)c_{n-k}\right) - 3|A_{16}|c_{n-17}.
\end{align*}
Using the values found in the table in Table~\ref{fig:321-1}, we have the desired result.
\end{proof}

Let us note that using using only $k=1$ from Table~\ref{fig:321-1}, we have $c_n\geq 2c_{n-1}$. Using this fact together with the result in Theorem~\ref{thm:321}, we obtain
\begin{align*}
    c_n \geq& \ 2c_{n-1}+c_{n-2}+ 4c_{n-4}+5c_{n-5}+68c_{n-7} +82c_{n-8} +262c_{n-9}
    +1384c_{n-10}\\&+3052c_{n-11}+11074c_{n-12}+39322c_{n-13}+120558c_{n-14}+167725c_{n-15},
\end{align*} which improves the lower bound $c_n\geq 2c_{n-1},$ proven by B\'{o}na and Cory in \cite{BC2019}.

\section{123-avoiding cycles}
In this section, we build permutations in $C_n(123)$ by inserting elements to both the beginning and the end of smaller cyclic 123-avoiding permutations. In some sense, we are wrapping one cyclic 123-avoiding permutation ``around'' another.
We begin with a binary operation that does exactly this for 123-avoiding cyclic permutations with certain properties.

\begin{definition}\label{defn:123}
    Let $\alpha=\alpha_1\alpha_2\cdots\alpha_{2m} \in \C_{2m}(123)$, and $\beta=\beta_1\beta_2\cdots\beta_{n} \in \C_{n}(123)$ for $n \geq 2$ and $m \geq 2$. Further suppose $\{\alpha_1, \alpha_2, \ldots, \alpha_m\} = [m+1, 2m]$. Define the permutation $\alpha \star \beta \in \S_{n + 2m-2}$ by
    \[ \alpha \star \beta = \bar{\alpha}_1\bar{\alpha}_2\cdots\bar{\alpha}_{m-1}\bar{\beta}_1\bar{\beta}_2\cdots\bar{\beta}_n\alpha_{m+2}\alpha_{m+3}\cdots\alpha_{2m}\]
    where
    \[ \bar{\alpha}_i = \alpha_i + n-2\]
    and
    \[ \bar{\beta_i} = \begin{cases}
    \alpha_m + n-2 & \text{if } \beta_i = n,\\
    \alpha_{m+1} & \text{if } \beta_i = 1,\\
    \beta_i + m-1 & \text{else.}
    \end{cases}\]

\end{definition}



We can think of this operation as follows. Given a cyclic 123-avoiding permutation $\alpha$ of even length, if the first half of $\alpha$ contains only elements larger than the second half of $\alpha$, we can consider all but the middle two elements of $\alpha$. Given any other cyclic 123-avoiding permutation $\beta$, we can insert the second half of $\alpha$ (minus one element) at the end of $\beta$ and the first half of $\alpha$ (minus one element) at the beginning of $\beta$, adjusting the values of $\beta$ as necessary. It is clear that the resulting permutation remains 123-avoiding. It turns out that it is also cyclic as is illustrated in the next example and proven in the subsequent lemma.

\begin{example} Let $\alpha = 65873214 = (1,6,2,5,3,8,4,7) \in \C_8(123)$ and $\beta = 462531 =(1,4,5,3,2,6) \in \C_6(123)$. Then
\[ \alpha \star \beta = \fbox{10\ 9\ 12}\   7\ 11\ 5\ 8\ 6\ 3\  \fbox{2\ 1\ 4}\]
where the elements from $\alpha$ are in boxes. In cyclic notation, we have
\[ \alpha \star \beta = (\fbox{1, 10, 2, 9,} 3, \fbox{12, 4,} 7, 8, 6, 5, 11).\]
This example is illustrated in Figure~\ref{fig:123 example}.
\end{example}


\begin{figure}
\centering
\begin{tabular}{c}
  $\alpha=65873214$ \\ \begin{tabular}{c}
\begin{tikzpicture}[scale=.5]
\newcommand\myx[1][(0,0)]{\pic at #1 {myx};}
\tikzset{myx/.pic = {\draw 
(-2.5mm,-2.5mm) -- (2.5mm,2.5mm)
(-2.5mm,2.5mm) -- (2.5mm,-2.5mm);}}
\draw[gray] (0,0) grid (8,8);
\foreach \x/\y in {
1/6,
2/5,
3/8,
4/7,
5/3,
6/2,
7/1,
8/4
}
\myx[(\x-.5,\y-.5)];
\draw[-, ultra thick,green!70!black] (.5,.5)--(.5,5.5)
--(5.5,5.5)--(5.5,1.5)
--(1.5,1.5)--(1.5,4.5)--(4.5,4.5)--(4.5,2.5)--(2.5,2.5)--(2.5,7.5)--(7.5,7.5)--(7.5,3.5)--(3.5,3.5)--(3.5,6.5)--(6.5,6.5)--(6.5,.5)--(.5,.5);
\end{tikzpicture}
\end{tabular} \\[2cm] $\beta = 462531$ \\
\begin{tabular}{c}
\begin{tikzpicture}[scale=.5]
\newcommand\myx[1][(0,0)]{\pic at #1 {myx};}
\tikzset{myx/.pic = {\draw 
(-2.5mm,-2.5mm) -- (2.5mm,2.5mm)
(-2.5mm,2.5mm) -- (2.5mm,-2.5mm);}}
\draw[gray] (0,0) grid (6,6);
\foreach \x/\y in {
1/4,
2/6,
3/2,
4/5,
5/3,
6/1
}
\myx[(\x-.5,\y-.5)];
\draw[-, ultra thick,red!70!black] (.5,.5)--(.5,3.5)
--(3.5,3.5)--(3.5,4.5)--(4.5,4.5)--(4.5,2.5)--(2.5,2.5)--(2.5,1.5)--(1.5,1.5)--(1.5,5.5)--(5.5,5.5)--(5.5,.5)--(.5,.5);
\end{tikzpicture} \end{tabular}
\end{tabular}
$\longrightarrow$
\begin{tabular}{c}
$\alpha\star\beta = 10 \ 9 \ 12 \ 7 \ 11 \ 5 \ 8 \ 6 \ 3 \ 2\ 1 \ 4$ \\
\begin{tikzpicture}[scale=.5]
\newcommand\myx[1][(0,0)]{\pic at #1 {myx};}
\tikzset{myx/.pic = {\draw 
(-2.5mm,-2.5mm) -- (2.5mm,2.5mm)
(-2.5mm,2.5mm) -- (2.5mm,-2.5mm);}}
\draw[gray] (0,0) grid (12,12);
\foreach \x/\y in {
1/10,
2/9,
3/12,
4/7,
5/11,
6/5,
7/8,
8/6,
9/3,
10/2,
11/1,
12/4
}
\myx[(\x-.5,\y-.5)];
\draw[-, ultra thick,green!70!black] 
(.5,.5)--(.5,9.5)--(9.5,9.5)--(9.5,1.5)--(1.5,1.5)--(1.5,8.5)--(8.5,8.5)--(8.5,2.5)--(2.5,2.5)--(2.5,11.5)--(11.5,11.5)--(11.5,3.5)--(3.5,3.5);
\draw[-, ultra thick,red!70!black] (3.5,3.5)--(3.5,6.5)--(6.5,6.5)--(6.5,7.5)--(7.5,7.5)--(7.5,5.5)--(5.5,5.5)--(5.5,4.5)--(4.5,4.5)--(4.5,8);
\draw[-, ultra thick,green!70!black] (4.5,8)--(4.5,10.5)--(10.5,10.5)--(10.5,.5)--(.5,.5);
\end{tikzpicture} 
\end{tabular}
\caption{This example demonstrates the operation defined in Definition~\ref{defn:123}. Here, $\beta\in\C_6(123)$ and $\alpha\in\C_8(123)$ with the first four elements of $\alpha$ all greater than the second four elements of $\alpha$. Thus we can compute $\alpha\star\beta$, which places the large elements of $\alpha$ in front and the small elements of $\alpha$ at the end, with $\beta$ in the middle (with two of the elements from both $\alpha$ and $\beta$ ``merged'' to keep it a cycle). Thus $\alpha$ is ``wrapped around'' $\beta$ and the result is a 123-avoiding cyclic permutation of length 12.}
\label{fig:123 example}
\end{figure}
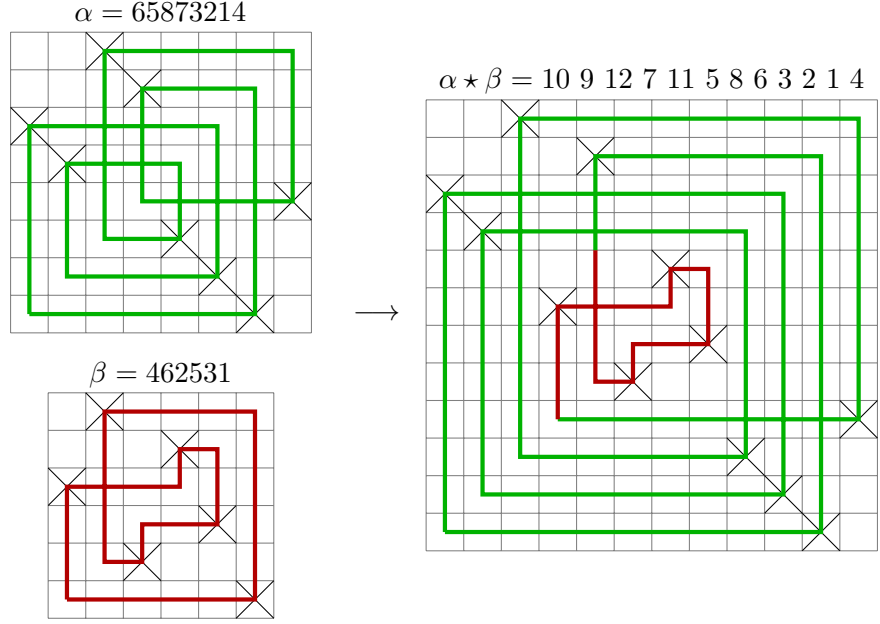

\begin{lemma} \label{lem:123-take2}
    Let $\alpha=\alpha_1\alpha_2\cdots\alpha_{2m} \in \C_{2m}(123)$, and $\beta=\beta_1\beta_2\cdots\beta_{n} \in \C_{n}(123)$ for $n \geq 1$ and $m \geq 2$. Further suppose $\{\alpha_1, \alpha_2, \ldots, \alpha_m\} = [m+1, 2m]$. Then $\alpha \star \beta \in \C_{n+2m-2}(123).$
\end{lemma}

\begin{proof} It is enough to show that $\alpha \star \beta$ is cyclic as it is 123-avoiding by construction. Consider the cycle notation of $\alpha$ and suppose first that $m+1$ comes before $m$ in the cycle notation occurring in positions $r$ and $s$, respectively:
\[ \alpha = (1, a_2, a_3, \ldots, a_{r-1}, m+1, a_{r+1}, \ldots, a_{s-1}, m, a_{s+1}, \ldots, a_{2m}). \]
Also consider the cycle notation of $\beta$ where $n$ is in position $t$:
\[ \beta = (1, b_2, \cdots, b_{t-1}, n, b_{t+1}, \cdots, b_n).\]
We claim that the cycle notation of $\alpha \star \beta$ is
\[ \alpha \star \beta = (1, \overline{a_2}, \cdots, \overline{a_r}, \overline{b_{t+1}}, \cdots, \overline{b_{n}}, \alpha_{m+1}, \overline{a_{r+2}}, \cdots, \overline{a_{s-1}}, \overline{b_2}, \cdots, \overline{b_{t-1}}, \alpha_m+n-2, \overline{a_{s+2}}, \cdots \overline{a_{2m}})\]
where $\overline{a_i} = a_i + n-2$     and
    \[ \bar{b_i} = \begin{cases}
    \alpha_m + n-2 & \text{if } b_i = n,\\
    \alpha_{m+1} & \text{if } b_i = 1,\\
    b_i + m-1 & \text{else.}
    \end{cases}\]
To prove our claim, we first note that in $\alpha \star \beta$, the elements in positions in the set $[1, m-1]$ are $[n+m-1, n+2m-2] \setminus \{\alpha_m + n-2\}$ while the elements in positions in the set $[n+m, n+2m-2]$ are $[1,m]\setminus \{\alpha_{m+1}\}$. Thus, the sequences
\[ 1, \overline{a_2}, \ldots, \overline{a_r},\]
\[ \overline{a_{r+1}}, \overline{a_{r+2}}, \ldots, \overline{a_{s-1}},\]
and \[ \overline{a_{s+1}}, \overline{a_{s+2}}, \ldots, \overline{a_{2m}} \]
will all appear in the cycle structure of $\alpha \star \beta$. Furthermore, $\overline{a_r} = a_r +n-2 = m+n-1.$ The element in position $m+n-1$ in $\alpha \star \beta$ is $\overline{\beta_n}$ Since $\beta_n = b_{t+1}$ (from the cycle structure of $\beta$), we see that $\overline{b_{t+1}}$ is in position $m+n-1$ as desired. Now note that $\overline{b_n}$ is the position in $\alpha \star \beta$ where the element 1 from $\beta$ (now $\alpha_{m+1}$) is. Thus, $\alpha_{m+1}$ follows $\overline{b_n}$ in the cycle structure of $\alpha \star \beta$. Similar arguments show that $\overline{b_2}$ is in position $\overline{a_{s-1}}$ and that $\overline{a_{s+2}}$ is in position $\overline{b_t} = \alpha_m+n-2$ thus showing that $\alpha \star \beta$ is cyclic.

Now suppose that in the cycle notation of $\alpha$, the element $m$  comes before the element $m+1$, and they occur in positions $r$ and $s$, respectively:
\[ \alpha = (1, a_2, a_3, \ldots, a_{r-1}, m, a_{r+1}, \ldots, a_{s-1}, m+1, a_{s+1}, \ldots, a_{2m}). \]
In this case, a similar argument shows the cycle notation of $\alpha \star \beta$ is
\[ \alpha \star \beta = (1, \overline{a_2}, \cdots, \overline{a_r}, \overline{b_2}, \cdots, \overline{b_{t-1}}, \alpha_m+n-2, \overline{a_{r+2}}, \cdots, \overline{a_{s}},\overline{b_{t+1}}, \cdots, \overline{b_{n}}, \alpha_{m+1},   \overline{a_{s+2}}, \cdots \overline{a_{2m}})\]
and our result holds.
\end{proof}

In order to bound the growth rate of $c_n(123)$, we will, for small $m$, count the number of permutations in $\C_{2m}(123)$ where the first $m$ elements are always larger than the last $m$ elements. For example, for $m =2$, there are 2 such permutations in $\C_4(123)$, namely 4312 and 3421. By Lemma~\ref{lem:123-take2}, if $\beta \in \C_{n-2}(123)$, then both $4312 \star \beta$ and $3412 \star \beta$ will be in $\C_n(123)$. They are also unique because one of the resulting permutations begins with $n$ while the other begins with $n-1$. We also note that if $\beta' \neq \beta \in \C_{n-2}(123),$ then $4312 \star \beta \neq 4312 \star \beta'$. Thus we have the bound:
\[ c_n(123) \geq 2c_{n-2}(123).\]

Although this is not a very spectacular bound (the growth rate is only $\sqrt{2}$), we can extend this process. Before doing so, we introduce some notation. First we define $A_1$ to be the set of desired permutations in $\C_4(123)$ as discussed above. That is,
\[ A_1 = \{4312, 3421\}.\] We then define $X_{1,n}$ to be the set of permutations formed by using permutations in $A_1$ to build new ones. Formally, we define
\[ X_{1,n} = \{\alpha \star \beta: \alpha \in A_1, \beta \in \C_{n-2}(123)\}.\]
As an example, we have\[ X_{1,6} = \{654132, 645312, 635142, 641532, 564231, 546321, 536241, 542631\},\] 
and in general, we see that $|X_{1,n}| = 2c_{n-2}.$

Moving forward, there are eight permutations in $\C_6(123)$ with the first three elements always larger than the last three elements:
\[ 654132, 654213, 645312, 564231, 465213, 465321, 546132, 546321.\]
All of these permutations can be used to build larger permutations.
However, when defining $A_2$, we want to be careful to avoid any overlap. For instance, $654132 \star 231 = 7645132$ while $4312 \star 53412 = 7645132$ as well. In fact, of the eight permutations listed above, only those that are not already in $X_{1,6}$  will produce unique permutations. To this end, we define $A_2$ in a more careful way. We first define $A_2'$ to be the set containing the eight permutations in $\C_6(123)$ with the first three elements larger than the last three,
\[ A_2' = \{\alpha \in \C_6(123): \{\alpha_1, \alpha_2, \alpha_3\} = \{4,5,6\}\}, \]and then 
\[ A_2 = \{\alpha \in \C_6(123): \alpha \in A_2' \text{ and } \alpha \notin X_{1,6}\}. \] Thus we have that $A_2 = \{654213, 465213, 465321, 546132\}$. We continue by defining
\[ X_{2,n} =  \{\alpha \star \beta: \alpha \in A_2, \beta \in \C_{n-4}(123)\}.\]
As an example, consider $n=7$. We know that $\C_3(123) = \{312, 231\}$, and thus

\begin{align*}
 X_{2,7} &= \{654213 \star 312, 465213 \star 312, 465321 \star 312, 546132 \star 312\} \cup \\
 & \quad \ \{654213 \star 213, 465213\star 213, 465321 \star 213,  546132 \star 213\}\\
 &= \{7652413, 5762413, 6752413, 6573421, 7645213, 5746213, 6745213, 6547321\}.
\end{align*}

In general, we have $|X_{2,n}| = 4c_{n-4}$, and $|X_{1,n} \cup X_{2,n}| = 2c_{n-2} + 4c_{n-4}.$ This process continues to generalize and we formalize the definitions of these sets here as well as the counting results.

\begin{definition}\label{def:123-new} For $n \geq 4,$ let 
\[ A_1 = A_1'= \{4312, 3421\} \quad \text{and} \quad  X_{1,n} = \{\alpha \star \beta: \alpha \in A_1, \beta \in \C_{n-2}(123)\}. \]
For $k \geq 2,$ let
\[ A_k' = \{\alpha \in \C_{2k+2}(123): \{\alpha_1, \ldots, \alpha_{k+1}\} = [k+2,2k+2]\}, \]
\[ A_k = \{\alpha \in A_k' : \alpha \notin X_{1, 2k+2} \cup X_{2, 2k+2} \cup \cdots \cup X_{k-1, 2k+2}\}, \]
and for $n \geq 2k+1$,
\[ X_{k,n} = \{\alpha \star \beta: \alpha \in A_k, \beta \in \C_{n-2k}(123)\}.\]
\end{definition}

\begin{lemma}\label{lem:123-X}
Suppose $A_k$, $A_k'$, and $X_{k,n}$ are as defined in Definition~\ref{def:123-new}. Then 
\begin{enumerate}
\item $X_{i,n} \cap X_{j,n} = \emptyset$ for $i \neq j$, and
\item $|X_{k,n}| = |A_k|c_{n-2k}$ for $k \geq 1$ where $|A_1| = 2$ and 
\[ |A_k| = |A_k'| - \sum_{i=1}^{k-1} |A_i||A_{k-i}'.|\]
for $k \geq 2.$
\end{enumerate}
\end{lemma}

\begin{proof}
First, let us note that $X_{i,n}\cap X_{j,n} = \emptyset$ for $i \neq j$ by construction. Indeed, if there were a permutation $\pi$ in the intersection, then we would have $\pi = \alpha\star\beta$ with $\alpha\in\C_{2i+2}(123)$ and $\pi=\alpha'\star\beta'$ with $\alpha'\in\C_{2j+2}(123)$. However, if $i<j$, then we would have that $\alpha'\in X_{i,2j+2}$, contradicting that $\alpha'\in A_j',$ which does not contain $X_{i,2j+2}$ for $i<j$.

Now let us consider the second statement. For $k=1$, we have $|X_{1,n}| = |A_1|c_{n-2}$ as seen in the discussion above. Now suppose $k >1$. Because $X_{i,n} \cap X_{j,n} = \emptyset,$ we have
\[ |A_k| = |A_k'| - \sum_{i=1}^{k-1} |A_k' \cap X_{i, 2k+2}|.\]
Thus, we need to show that $|A_k' \cap X_{i, 2k+2}| = |A_i||A_{k-i}'|.$ To this end, suppose $\pi \in A_k' \cap X_{i, 2k+2}$ for some $i \in [1, k-1].$ Since $\pi \in X_{i, 2k+2},$ there exist $\alpha \in A_i$ and $\beta \in \C_{2k+2-2i}$ with $\pi = \alpha \star \beta.$ Furthermore, since $\pi \in A_k'$, we know the first $k+1$ entries in $\pi$ are in $[k+2, 2k+2].$ Combining these facts tells us the first $k+1-i$ elements of $\beta$ must be larger than the last $k+1-i$ elements and thus $\beta \in A_{k-i}$. Since every element $\pi$ can be written as $\alpha \star \beta$ with $\alpha \in A_i$ and $\beta \in A_{k-i} \subseteq A_{k-i}'$, we have  $|A_k' \cap X_{i, 2k+2}| \leq |A_i||A_{k-i}'|$.
Furthermore, given $\alpha\in A_i$ and $\beta\in A_{k-i}'$, we obtain an element of $X_{i,2k+2}$ with the first $k+1$ elements larger than the second $k+1$ elements, by definition of $\star$, giving us equality.
\end{proof}

\begin{table}
\begin{center}
\begin{tabular}{c|l|l|l}
$k$ & $|A_k'|$ & $\sum_{i=1}^{k-1} |A_i||A_{k-i}'|$ & $|A_k|$ \\ \hline
1 & 2 & 0 & 2  \\
2 & 8 & 4 &4 \\
3 & 44 & 24 & 20\\
4 & 296 & 160 & 136\\
5 & 2252 &1200& 1052\\
6 & 18,874 & 9760 & 9114\\
7 & 169,860 & 85,304 & 84,556 \\
8 & 1,616,942 & 788,824 & 828,118\\
9 &16,102,076&7,642,168&8,459,908
\end{tabular}
\caption{Small values of $A_k$ for Definition~\ref{def:123-new}. Here $A'_k$ are exactly those permutations $\pi$ in $\C_{2k+2}(123)$ with the property that the first $k+1$ elements are larger than the second $k+1$ elements and $A_k$ are those elements of $A'_k$ that cannot be built from the $\star$ operation.}
 \label{fig:123-half}
\end{center}
\end{table}

Small values of $A_k$ can be found by computer and are listed in Table~\ref{fig:123-half}. Using these small values, we can bound $c_n$ for $n \geq 15$ as seen in the main result here.

\begin{theorem}\label{thm:123} Let $n \geq 20$ and let $c_n = |\C_n(123)|$. Then
\[ c_n \geq 2c_{n-2} + 4c_{n-4} + 20c_{n-6} + 136c_{n-8} + 1052c_{n-10} + 9114c_{n-12} + 84556c_{n-14}+828118c_{n-16} + 8459908c_{n-18}\]
which has a growth rate of approximately $2.66612$.
\end{theorem}

\begin{proof} By Lemma~\ref{lem:123-take2}, we know that $X_{k,n} \subseteq \C_n(123).$ Furthermore, by construction, we have $X_{i,n} \cap X_{j,n} = \emptyset$ for $i \neq j.$ Thus,
\begin{align*}
 c_n &\geq \left|\bigcup_{k=1}^8 X_{k,n}\right| = \sum_{k=1}^8 |X_{k,n}| \\
  &= \sum_{k=1}^8 |A_{k}|c_{n-2k}. \end{align*}
Using small values of $|A_k|$ found in  Table~\ref{fig:123-half}, we have the desired result.
\end{proof}


\section{132-avoiding cycles}

In this section, we address the case of 132-avoiding cyclic permutations in a similar way to the previous section. 
We begin with a binary operation $\ostar$ that allows us to augment a 132-avoiding cyclic permutation by another 132-avoiding cyclic permutation that has some special properties.

\begin{definition}\label{defn:132}
    Let $\alpha=\alpha_1\alpha_2\cdots\alpha_{2m} \in \C_{2m}(132)$, and $\beta=\beta_1\beta_2\cdots\beta_{n} \in \C_{n}(132)$ for $n \geq 2$ and $m \geq 2$. Further suppose $\{\alpha_1, \alpha_2, \ldots, \alpha_m\} = [m+1, 2m]$. Define the permutation $\alpha \ostar \beta \in \S_{n + 2m-2}$ by
    \[ \alpha \ostar \beta = \bar{\alpha}_1\bar{\alpha}_2\cdots\bar{\alpha}_{m-1}\bar{\beta}_1\bar{\beta}_2\cdots\bar{\beta}_{n-1}\bar\alpha_{m+1}\bar\alpha_{m+2}\cdots\bar\alpha_{2m}\]
    where 
    \[ \bar{\alpha_i} = \begin{cases}
        \alpha_i+n-2 & i<m \\
        \beta_n +m-1& \alpha_i= m \\
        \alpha_i & i>m \text{ and $\alpha_i\neq m$}
    \end{cases}\]
    and
    \[ \bar{\beta_i} = \begin{cases}
        \beta_i+m-1 & \beta_i\neq n\\
        \alpha_m+n-2 & \beta_i=n.
    \end{cases} \]
\end{definition}
    In other words, we place the first $m-1$ elements of $\alpha$ at the front and the last $m$ elements (except $m$) at the end, placing $\beta_m$ where $m$ used to be, and scale everything appropriately. Let us see an example.

\begin{example}\label{ex:132} Let $\alpha = 654213 =(1,6,3,4,2,5) \in \C_6(132)$ and $\beta = 76821345 = (1,7,4,2,6,3,8,5) \in \C_8(132).$ Then
\[ \alpha \ostar \beta = \fbox{12 \ 11}\   9 \ 8 \ 10 \ 4 \ 3 \ 5 \ 6 \ \fbox{2\ 1}\ 7\]
where the elements from $\alpha$ are in boxes. In cyclic notation, we have
\[ \alpha \ostar \beta = (\fbox{1, 12}, 7,3,9,6,4,8,5,10, \fbox{2, 11}).\]
This example is illustrated in Figure~\ref{fig:132 example}.
\end{example}

\begin{figure}
\centering
\begin{tabular}{c}
  $\alpha=654213$ \\ \begin{tabular}{c}
\begin{tikzpicture}[scale=.5]
\newcommand\myx[1][(0,0)]{\pic at #1 {myx};}
\tikzset{myx/.pic = {\draw 
(-2.5mm,-2.5mm) -- (2.5mm,2.5mm)
(-2.5mm,2.5mm) -- (2.5mm,-2.5mm);}}
\draw[gray] (0,0) grid (6,6);
\foreach \x/\y in {
1/6,
2/5,
3/4,
4/2,
5/1,
6/3
}
\myx[(\x-.5,\y-.5)];
\draw[-, ultra thick,green!70!black] (.5,.5)--(.5,5.5)
--(5.5,5.5)--(5.5,2.5)
--(2.5,2.5)--(2.5,3.5)--(3.5,3.5)--(3.5,1.5)--(1.5,1.5)--(1.5,4.5)--(4.5,4.5)--(4.5,.5)--(.5,.5);
\end{tikzpicture}
\end{tabular} \\[2cm] $\beta = 76821345$ \\
\begin{tabular}{c}
\begin{tikzpicture}[scale=.5]
\newcommand\myx[1][(0,0)]{\pic at #1 {myx};}
\tikzset{myx/.pic = {\draw 
(-2.5mm,-2.5mm) -- (2.5mm,2.5mm)
(-2.5mm,2.5mm) -- (2.5mm,-2.5mm);}}
\draw[gray] (0,0) grid (8,8);
\foreach \x/\y in {
1/7,
2/6,
3/8,
4/2,
5/1,
6/3,
7/4,
8/5
}
\myx[(\x-.5,\y-.5)];
\draw[-, ultra thick,red!70!black] (.5,.5)--(.5,6.5)
--(6.5,6.5)--(6.5,3.5)--(3.5,3.5)--(3.5,1.5)--(1.5,1.5)--(1.5,5.5)--(5.5,5.5)--(5.5,2.5)--(2.5,2.5)--(2.5,7.5)--(7.5,7.5)--(7.5,4.5)
--(4.5,4.5)--(4.5,.5)--(.5,.5);
\end{tikzpicture} \end{tabular}
\end{tabular}
$\longrightarrow$
\begin{tabular}{c}
$\alpha\ostar\beta = 12 \ 11 \ 9 \ 8 \ 10 \ 4 \ 3 \ 5 \ 6 \ 2\ 1 \ 7$ \\
\begin{tikzpicture}[scale=.5]
\newcommand\myx[1][(0,0)]{\pic at #1 {myx};}
\tikzset{myx/.pic = {\draw 
(-2.5mm,-2.5mm) -- (2.5mm,2.5mm)
(-2.5mm,2.5mm) -- (2.5mm,-2.5mm);}}
\draw[gray] (0,0) grid (12,12);
\foreach \x/\y in {
1/12,
2/11,
3/9,
4/8,
5/10,
6/4,
7/3,
8/5,
9/6,
10/2,
11/1,
12/7
}
\myx[(\x-.5,\y-.5)];
\draw[-, ultra thick,green!70!black] 
(.5,.5)--(.5,11.5)--(11.5,11.5)--(11.5,10)--(11.5,6.5) -- (10,6.5);
\draw[-, ultra thick,red!70!black] (10,6.5)--(6.5,6.5)--(6.5,2.5)--(2.5,2.5)--(2.5,8.5)--(8.5,8.5)--(8.5,5.5)--(5.5,5.5)--(5.5,3.5)--(3.5,3.5)--(3.5,7.5)--(7.5,7.5)--(7.5,4.5)--(4.5,4.5)--(4.5,9.5)--(9.5,9.5)--(9.5,3);
\draw[-, ultra thick,green!70!black](9.5,3)--(9.5,1.5)--(1.5,1.5)--(1.5,10.5)--(10.5,10.5)--(10.5,.5)--(.5,.5);
\end{tikzpicture} 
\end{tabular}
\caption{This example demonstrates the operation defined in Definition~\ref{defn:132}. Here $\beta\in\C_8(132)$ and $\alpha\in\C_6(132)$ with the first three elements of $\alpha$ all greater than the second three elements of $\alpha$. Thus we can compute 
$\alpha\,$\circled[0]{$\star$}$\,\beta$, which places the large elements of $\alpha$ in front and the small elements of $\alpha$ at the end, with $\beta$ in the middle 
(with two of the elements from both $\alpha$ and $\beta$ ``merged'' to keep it a cycle). Thus $\alpha$ is ``wrapped around'' $\beta$ and the result is a 132-avoiding cyclic permutation of 
length 12. }
\label{fig:132 example}
\end{figure}

Notice that in Example~\ref{ex:132}, we have that $\alpha\ostar\beta$ is still cyclic and avoids 132. In fact, this is always the case, as seen in the next lemma.

\begin{lemma}\label{lem:132-cyc}
    Let $m\geq 2$, $n\geq 2$. Let $\alpha=\alpha_1\ldots\alpha_{2m}\in\C_{2m}(132)$ and $\beta\in\C_n(132)$ with the additional condition that $\{\alpha_1,\ldots, \alpha_{m}\}=[m+1,2m].$ Then $\alpha\ostar\beta\in\C_{n+2m-2}(132).$
\end{lemma}
\begin{proof}
    It is clear from the definition that $\alpha\ostar\beta$ avoids 132 if both $\alpha$ and $\beta$ do. It remains to show that $\alpha\ostar\beta$ is cyclic. Let
    \[\alpha=(1,a_2,\ldots, a_r, m, a_{r+2},\ldots, a_{2m})\] and 
    \[
    \beta=(1,b_2,\ldots, b_s,n,b_{s+2},\ldots, b_n)
    \]
    and consider the permutation 
    \[
    \gamma = (1,a_2',\ldots, a_r',b'_{s+2}, \ldots b'_n,b_1',b'_2\ldots, b'_{s},a'_{r+2},a'_{r+3}, \ldots, a'_{2m}) 
    \]
    where $b_i'=b_i+m-1$ and $a_i'=a_i$ if $a_i<m$ and $a_{i}'=a_i+n-2$ if $a_i>m$.
     We will show that $\gamma$ is exactly $\alpha\ostar\beta$ and thus $\alpha\ostar\beta$ is cyclic.
     First, note that $b'_{s+2} = \beta_n+m-1$ and $a'_{r+2} = \alpha_m+n-2$.  Note that if $\alpha_i=m$, then $\gamma_{n+i-2}=\beta_{n+m-1}$ and that if $\beta_j=n$, then $\gamma_{j+m-1} = \alpha_m+n-2$, as desired. Furthermore, $b_i\in[m,m+n-2]$ and $a_i \in [1,m-1]\cup[m+n-1,2m+n-2]$ (with $a_1=1$), so we have $\gamma_j=\beta_{j+m-1}+m-1$ for all $j$ with $\beta_j\neq n$, and similarly for the values of $\alpha,$ appropriately scaled. Thus $\gamma=\alpha\ostar\beta$, by definition.
\end{proof}

Similar to the previous section, we define sets $A_k$, $A'_k,$ and $X_{k,n}$. In particular, let 
\[
A_1=\{4312, 3421\}
\] and let 
\[
X_{1,n} = \{\alpha\ostar\beta : \alpha\in A_1, \beta\in\C_{n-2}(132)\}.
\]
Therefore $X_{1,n}$ are exactly those permutations formed by $4312\ostar\beta$, which involves taking a permutation $\beta\in\C_{n-2}(132)$ and inserting $n$ at the beginning and $1$ before the last spot, and $3421\ostar\beta$, which involves taking a permutation $\beta\in\C_{n-2}(132)$ and inserting 1 in the last spot and $n-1$ into the first spot. 
For example, 
\[X_{1,6} = \{652413, 654213,645312,634512, 562431,564231,546321,534621\}. \]
In general, since it is clear that for $\alpha, \alpha'\in A_1$, $\alpha\ostar\beta$ and $\alpha'\ostar\beta'$ are only equal if $\alpha=\alpha'$ and $\beta=\beta'$, we must have $|X_{1,n}|=2c_{n-2}.$ As with the case when $\sigma=123$, when moving forward, we want to avoid overlap. For this reason, we define 
\[A_k' = \{\alpha\in\C_{2k+2} : \{\alpha_1, \ldots, \alpha_{k+1}\} =[k+2,2k+2]\}\]
 and 
\[A_k = \bigg\{\alpha\in\C_{2k+2} : 
\alpha\in A_k' \text{ and } \alpha\not\in \bigcup_{j<k}X_{j,2k+k}\bigg\}.\]
where
\[
X_{k,n} = \{\alpha\ostar\beta : \alpha\in A_k, \beta\in\C_{n-2k}(132)\}.
\]
This leads us to the following lemma. Because the proof of this lemma is essentially identical to the proof of Lemma~\ref{lem:123-X}, we omit it here.
\begin{lemma}\label{lem:132-Ak} For $n \geq 4$ and $k \geq 1,$ let $A'_k$, $A_k$, and $X_{k,n}$ be defined as above.
Then, $|X_{k,n}| = |A_k|c_{n-2k}(132)$ for $k \geq 1$ where
\[ |A_k| = |A_k'| - \sum_{i=1}^{k-1} |A_i||A_{k-i}'|.\]
\end{lemma}

\begin{table}
\begin{center}
\begin{tabular}{c|l|l|l}
$k$ & $|A_k'|$ & $\sum_{i=1}^{k-1} |A_i||A_{k-i}'|$ & $|A_k|$ \\ \hline
1 & 2 & 0 & 2  \\
2 & 8 & 4 &4 \\
3 & 44 & 24 & 20\\
4 & 294 & 160 & 134\\
5 & 2242 &1196& 1046\\
6 & 18,800 & 9704 & 9096\\
7 & 169,436 & 84,904 & 84,532 \\ 
8 &  1,616,070 & 786,164 & 829,906 \\
\end{tabular}
\caption{Small values of $A_k$ for Lemma~\ref{lem:132-Ak}. Here $A'_k$ are exactly those permutations $\pi$ in $\C_{2k+2}(132)$ with the property that the first $k+1$ elements are larger than the second $k+1$ elements and $A_k$ are those elements of $A'_k$ that cannot be built from the \circled[0]{$\star$} operation.}
 \label{fig:132-half}
\end{center}
\end{table} 

Small values of $A_k$ can be  computed directly and are listed in Table~\ref{fig:132-half}. Using these small values, we can bound $c_n(132)$ for $n \geq 15$ as seen in the main result here.

\begin{theorem}\label{thm:132} Let $n \geq 18$ and let $c_n = |\C_n(132)|$. Then
\[ c_n \geq 2c_{n-2} + 4c_{n-4} + 20c_{n-6} + 134c_{n-8} + 1046c_{n-10} + 9096c_{n-12} + 84532c_{n-14} + 829906c_{n-16}\]
which has a growth rate of approximately 2.60078.
\end{theorem}

\begin{proof} By Lemma~\ref{lem:132-cyc}, we know that $X_{k,n} \subseteq \C_n(132).$ Furthermore, by construction, we have $X_{i,n} \cap X_{j,n} = \emptyset$ for $i \neq j.$ Thus,
\begin{align*}
 c_n &\geq \left|\bigcup_{k=1}^8 X_{k,n}\right| = \sum_{k=1}^8 |X_{k,n}| \\
  &= \sum_{k=1}^8 |A_{k}|c_{n-2k}. \end{align*}
Using small values of $|A_k|$ found in  Table~\ref{fig:132-half}, we have the desired result.
\end{proof}

Finally, let us notice that doing this procedure does not prove B\'{o}na and Cory's conjecture that $c_n(132)\geq 2c_{n-1}(132).$ However, we can augment a single cyclic permutation avoiding 132 in the middle (instead of the front and end of the permutation) to obtain two distinct cyclic 132-avoiding permutations. This process is also easily reversible, and thus would prove that $c_n(132)\geq 2c_{n-1}(132).$

In order to determine exactly where to augment a given permutation, we define the (lower left) Dyck path associated to a 132-avoiding cyclic permutation.
In particular, we consider the bijection between 132-avoiding permutations and Dyck paths given in \cite{K01}, defined in the following way. For $\pi\in\S_n(132)$, we
define $D(\pi)$ to be the Dyck path obtained by reading the permutation $\pi$ left-to-right, with left-to-right minima in positions $\{i_1,\ldots, i_k\}.$ In this case, for the $j$th left-to-right minima, write $u^\ell d^m$ where $\ell = \pi_{i_j}-\pi_{i_{j-1}}$ (with the convention that $\pi_{i_0}=n+1$) and $m=i_{j+1}-i_j$ (with the convention that $i_{k+1}=n+1$). 
For example, if $\pi = 76821345$, then the left-to-right minima occur in positions $\{1,2,4,5\}$. Therefore $D(\pi)$ starts with $uud$ since $n+1-\pi_1=2$ and $2-1=1.$ This is followed by $udd$ since $\pi_1-\pi_2=7-6=1$ and $4-2=2$. This continues and we obtain
$D(\pi) = uududduuuududddd$. The Dyck path $D(\pi)$ can also be easily observed from the diagram of the permutation, illustrated in Figure~\ref{fig:Dyck}.

\begin{figure}
\centering
\begin{tabular}{cc}
\begin{tabular}{c}
$\pi = 76821345$ \\
\begin{tikzpicture}[scale=.5]
\newcommand\myx[1][(0,0)]{\pic at #1 {myx};}
\tikzset{myx/.pic = {\draw 
(-2.5mm,-2.5mm) -- (2.5mm,2.5mm)
(-2.5mm,2.5mm) -- (2.5mm,-2.5mm);}}
\draw[gray] (0,0) grid (8,8);
\foreach \x/\y in {
1/7,
2/6,
3/8,
4/2,
5/1,
6/3,
7/4,
8/5
}
\myx[(\x-.5,\y-.5)];
\draw[-, ultra thick,cyan!70!black] (0,8)--(0,6)--(1,6)--(1,5)--(3,5)--(3,1)--(4,1)--(4,0)--(8,0);
\draw[dashed, ultra thick,magenta!70!black] (0,0)--(8,8);
\end{tikzpicture} \end{tabular} 
&
\begin{tabular}{c}
$p(D_L(\pi)) = 8  7  9   5   2  1 3 4 5$ \\
\begin{tikzpicture}[scale=.5]
\newcommand\myx[1][(0,0)]{\pic at #1 {myx};}
\tikzset{myx/.pic = {\draw 
(-2.5mm,-2.5mm) -- (2.5mm,2.5mm)
(-2.5mm,2.5mm) -- (2.5mm,-2.5mm);}}
\draw[gray] (0,0) grid (9,9);
\foreach \x/\y in {
1/8,
2/7,
3/9,
4/5,
5/2,
6/1,
7/3,
8/4,
9/6
}
\myx[(\x-.5,\y-.5)];
\draw[-, ultra thick,cyan!70!black] (0,9)--(0,7)--(1,7)--(1,6)-- (3,6)--(3,5);
\draw[-, ultra thick,orange!90!black] (3,5)
--(3,4)--(4,4);
\draw[-, ultra thick,cyan!70!black] (4,4)--(4,1)--(5,1)--(5,0)--(9,0);
\draw[dashed, ultra thick,magenta!70!black] (0,0)--(9,9);
\end{tikzpicture} 
\end{tabular} 
\end{tabular} 
\caption{On the left, an example of a 132-avoiding permutation, $\pi = 76821345$, together with Dyck path $D(\pi)$. On the right, the permutation $p(D_L(\pi)) = 8 7  9   5   2  1 3 4 5$ associated to the Dyck path obtained by adding $ud$ in the correct place on the original Dyck path for $\pi$. The resulting permutation is still cyclic and still avoids the pattern 132.}
    \label{fig:Dyck}
\end{figure}
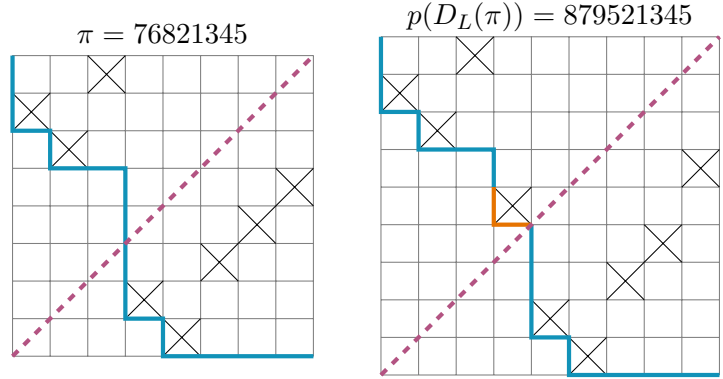

We let $p(D)$ be the 132-avoiding permutation associated to Dyck path D. 
Suppose for a Dyck path $D$, the height after the $k$-th down step is $h_k.$ Then we can find $\pi=p(D)$ by noting that the number of elements after $\pi_k$ that are greater than it is equal to $h_k$. For example, if $D=uuuudduduudddd$, then $h_1=3$, $h_2=2,$ $h_3=2, h_4=3,h_5=2, h_6=1,$ and $h_7=0,$ so $\pi_1=4$ since there must be three elements to the right of it greater than it. Similarly, $\pi_2=5$ since there must be two elements (6 and 7) to its right. Filling out the rest, we get $\pi = 4531267.$ 

Instead of augmenting a 132-avoiding permutation directly, we will augment its associated Dyck path by adding $ud$ to a particular position. First, let us see what this does to the corresponding permutation. Let $\pi\in\S_n(132)$ with $D=D(\pi)$ and let $D'$ be obtained by inserting $ud$ into $D$. Then if $m$ $d$'s appear before this insertion and $r$ $u$'s appear after this insertion, it is straightforward to check that $\pi' = p(D')$ is obtained from $\pi$ by inserting the element $r$ into $\pi$ after position $m$.

Now, define $D_L(\pi)$ to be the Dyck path defined by inserting $ud$ after the $(n-1)$st element of $D(\pi)$. Define $D_R(\pi)$ to be the Dyck path defined by inserting $ud$ after the $(n+1)$th element of $D(\pi)$. Via the bijection described above, it is clear that these Dyck paths will still be associated to a 132-avoiding permutation; it is necessary for us to show that if the original Dyck path was associated to a cyclic permutation, then the augmented Dyck path is as well. 
    However, in the case of $D_L(\pi)$, this amounts to inserting $j+1$ into position $j$ of the one-line form. In this case, the cycle form changes only by $j+1$ being inserted after $j$ in the cycle form and nothing else changes. Similarly, in the case of $D_R(\pi)$, $j-1$ is inserted into position $j,$ thus keeping the permutation cyclic.

Finally, for any $\pi, \tau \in \C_n(132)$, we cannot have $D_L(\pi) = D_R(\tau).$
    If this were possible, then we would have a $ud$ appear on either side of the line $y=x$ in the diagram of the permutation, implying there is a 2-cycles in $\pi$, contradicting that the permutation is cyclic.

\begin{theorem}\label{thm:132-dyck}
    For $n\geq 3$, \[c_n(132)\geq 2c_{n-1}(132)\]
\end{theorem}

\begin{proof}
Let us augment $\pi\in\C_{n-1}(132)$ by considering $\pi'=p(D_L(\pi))$ and $\pi''=p(D_R(\pi))$. It is clear that there are $c_{n-1}(132)$ for each of left and right augmentation and as described above, we know that there is no intersection. Since the process is easily invertible, the result follows.
\end{proof}

Finally, we note that the process described here of adding elements to the ``middle'' of the permutation can be extended, but it is less clear in this context how to track those permutations obtained by adding in the middle in two different ways (i.e., counting the intersection). Also, the Dyck path perspective can also be considered for $\sigma=123,$ but in this case, we must be careful when augmenting the Dyck path to avoid a 123. In particular, $ud$ could not be added between a $u$ and $d$ in the $\sigma=123$ case. For that reason, we cannot easily obtain $c_n(123)\geq 2c_{n-1}(123)$ using Dyck paths.

\section{Discussion}

The question of enumerating those cyclic permutations that avoid a single pattern of length 3 remains open, as does enumerating cyclic permtuations avoiding the pair $\{132,213\}$. The values of $c_n(\sigma)$ for $\sigma\in\S_3$ appear in Table~\ref{tab:discussion}. In this paper, we demonstrate that these permutation do exhibit some interesting structure which may make them amenable to enumeration, given the right techniques.

\begin{table}[h]
    \centering
    \begin{tabular}{c||c|c|c|c|c|c}
         $\sigma$& 123 & 132 & 213 & 231 & 312 & 321  \\ \hline\hline
         2& 1&1&1&1&1&1 \\ \hline
         3& 2&2&2&1&1&2 \\ \hline
         4& 4&4&4&2&2&4 \\ \hline
         5& 10&10&10&5&5&10 \\ \hline
         6& 24&24&24&12&12&24 \\ \hline
         7& 68&68&68&30&30&66 \\ \hline
         8& 188&182&182&86&86&178 \\ \hline
         9& 586&544&544&253&253&512 \\ \hline
         10& 1722&1574&1574&748&748&1486 \\ \hline
         11& 5492 & 4888 & 4888 & 2274 & 2274 & 4446 \\ \hline
         12 & 16924 &  14864 & 14864 &7152 & 7152  & 13468 \\ \hline
         13 & 55582 & 47610 & 47610 & 22890 & 22890 & 41648 \\ \hline
         OEIS & A309504 &A309505 & A309505& A309506& A309506& A309508
    \end{tabular}
    \caption{The values of $c_n(\sigma)$ for all $\sigma\in\S_3$ for $2\leq n \leq 13$. Values for $0\leq n \leq 24$ are listed on OEIS, provided by user Andrew Howroyd \cite{OEIS}.}
    \label{tab:discussion}
\end{table}

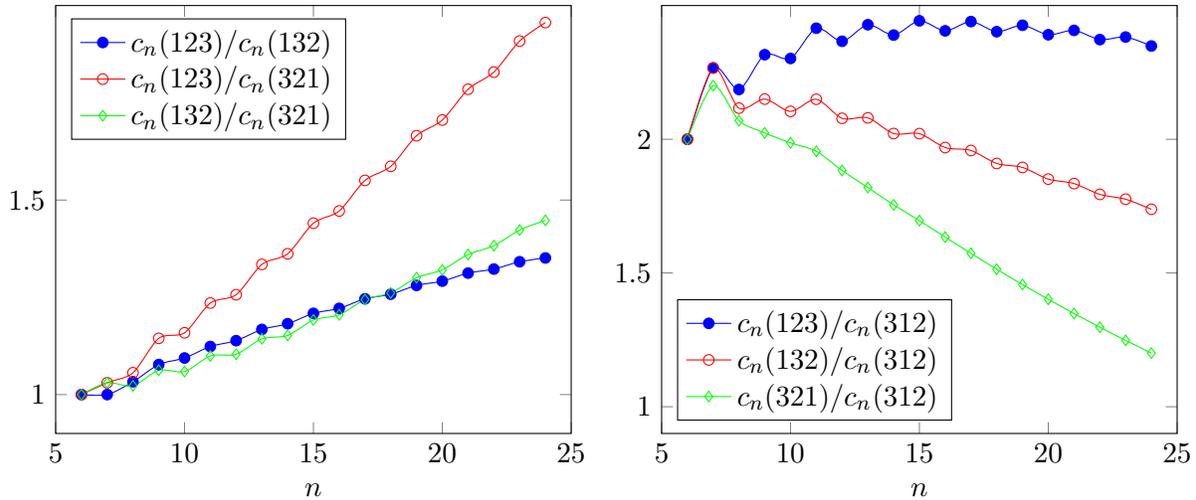
\begin{figure}
    \centering
    \begin{tikzpicture}
    \begin{axis}[
        xlabel=$n$,
        xmin=5, xmax=25,
        ymin=.9, ymax=2,
        xtick={5,10,15,20,25},
        ytick={0,.5,1,1.5},
        legend pos=north west
        ]
    \addplot[smooth,mark=*,blue] plot coordinates { 
        (6, 1.00000000000000) (7, 1.00000000000000) (8, 1.03296703296703) (9, 1.07720588235294) (10, 1.09402795425667) (11, 1.12356792144026) (12, 1.13858988159311) (13, 1.16744381432472) (14, 1.18213704622443) (15, 1.20873847605337) (16, 1.22156917669787) (17, 1.24618607932113) (18, 1.25799455129462) (19, 1.28053146409565) (20, 1.29159255477788) (21, 1.31221866376555) (22, 1.32259403857146) (23, 1.34141355973083) (24, 1.35114142989817)
    };
    \addlegendentry{$c_n(123)/c_n(132)$}

    \addplot[smooth,color=red,mark=o]
        plot coordinates {
            (6, 1.00000000000000) (7, 1.03030303030303) (8, 1.05617977528090) (9, 1.14453125000000) (10, 1.15881561238223) (11, 1.23526765632029) (12, 1.25660825660826) (13, 1.33456588551671) (14, 1.36181228779056) (15, 1.44052148205588) (16, 1.47188853186501) (17, 1.55051006205891) (18, 1.58664893154088) (19, 1.66530323091258) (20, 1.70562322577139) (21, 1.78480056497876) (22, 1.82882255094190) (23, 1.90869592852454) (24, 1.95615355015918)
        };
    \addlegendentry{$c_n(123)/c_n(321)$}

    \addplot[smooth,color=green,mark=diamond]
        plot coordinates {
            (6, 1.00000000000000) (7, 1.03030303030303) (8, 1.02247191011236) (9, 1.06250000000000) (10, 1.05921938088829) (11, 1.09941520467836) (12, 1.10365310365310) (13, 1.14315213215521) (14, 1.15199188803024) (15, 1.19175612474859) (16, 1.20491623392447) (17, 1.24420428681371) (18, 1.26125262618032) (19, 1.30047818238396) (20, 1.32055826697198) (21, 1.36013959735726) (22, 1.38275426745249) (23, 1.42289893722823) (24, 1.44777852774939)
        };
    \addlegendentry{$c_n(132)/c_n(321)$}
    \end{axis}
    \end{tikzpicture}
    \begin{tikzpicture}
    \begin{axis}[
        xlabel=$n$,
        xmin=5, xmax=25,
        ymin=.9, ymax=2.5,
        xtick={5,10,15,20,25},
        ytick={0,.5,1,1.5,2},
        legend pos=south west
        ]
    \addplot[smooth,mark=*,blue] plot coordinates { 
        (6, 2.00000000000000) (7, 2.26666666666667) (8, 2.18604651162791) (9, 2.31620553359684) (10, 2.30213903743316) (11, 2.41512752858399) (12, 2.36633109619687) (13, 2.42822193097422) (14, 2.38954561997062) (15, 2.44289929399775) (16, 2.40537027624323) (17, 2.44000163442662) (18, 2.40216032700127) (19, 2.42620577237709) (20, 2.39077840452060) (21, 2.40708972096117) (22, 2.37266402010130) (23, 2.38219124923146) (24, 2.34873287716982)
    };
    \addlegendentry{$c_n(123)/c_n(312)$}

    \addplot[smooth,color=red,mark=o]
        plot coordinates {
            (6, 2.00000000000000) (7, 2.26666666666667) (8, 2.11627906976744) (9, 2.15019762845850) (10, 2.10427807486631) (11, 2.14951627088830) (12, 2.07829977628635) (13, 2.07994757536042) (14, 2.02137783229320) (15, 2.02103212762285) (16, 1.96908232634470) (17, 1.95797535770568) (18, 1.90951568472945) (19, 1.89468657382077) (20, 1.85103142293334) (21, 1.83436631975175) (22, 1.79394731180252) (23, 1.77588129473619) (24, 1.73833236491522)
        };
    \addlegendentry{$c_n(132)/c_n(312)$}

    \addplot[smooth,color=green,mark=diamond]
        plot coordinates {
            (6, 2.00000000000000) (7, 2.20000000000000) (8, 2.06976744186047) (9, 2.02371541501976) (10, 1.98663101604278) (11, 1.95514511873351) (12, 1.88310961968680) (13, 1.81948449104412) (14, 1.75468061302889) (15, 1.69584370967609) (16, 1.63420681944944) (17, 1.57367674943467) (18, 1.51398351535043) (19, 1.45691530968059) (20, 1.40170371064180) (21, 1.34866033112770) (22, 1.29737246452878) (23, 1.24807268335976) (24, 1.20068942286187)
        };
    \addlegendentry{$c_n(321)/c_n(312)$}
    \end{axis}
    \end{tikzpicture}
    \caption{Here, we plot the ratio $c_n(\sigma)/c_n(\tau)$ for each pairs $\sigma,\tau$ with $c_n(\sigma)\geq c_n(\tau)$ for $5\leq n\leq 24.$ The cases with $\tau=312$ appear on the right.}
    \label{fig:inequalities}
\end{figure}

 B\'{o}na and Cory conjectured the following inequalities for all $n\geq 1$ in \cite{BC2019}:
\[c_n(123)\geq c_n(132)\geq c_n(321)\geq c_n(312).\]
We note that though these all hold for the values in Table~\ref{tab:discussion} and the corresponding OEIS entries, the last inequality becomes closer and closer as $n$ gets larger for the values of $n$ known to us ($n\leq 24$). For example, consider the relationships between these four values as $n$ grows, demonstrated in Figure~\ref{fig:inequalities}.
Though the inequalities $c_n(123)\geq c_n(132)\geq c_n(321)$ seem to be consistent with this data, it seems feasible than the inequalities involving $c_n(312)$ will not hold. In fact, it is possible that there is a large enough $n$ for which $c_n(312)\geq c_n(\sigma)$ for any other $\sigma\in\{123,132,213,321\}.$ If the relationship between $c_n(312)$ and $c_n(321)$ continues as in Figure~\ref{fig:inequalities}, it may even be the case that $c_n(312)>c_n(321)$ for some $n$ close to 28.
It would be interesting to further investigate these inequalities.

\subsection*{Disclaimer}
The views expressed in this article do not necessarily represent 
the views or opinions of the U.S. Naval Academy, Department of the Navy, or Department of Defense or any of its components.

\bibliographystyle{amsplain}

\end{document}